\documentclass[a4paper, twoside,10pt]{article}
\usepackage{amsmath}
\usepackage{amssymb}
\usepackage{amsfonts}
\usepackage{graphicx}
\usepackage{subfigure}
\usepackage{eucal}
\usepackage{amsthm}
\usepackage{verbatim}
\usepackage{epsfig}
\usepackage{hyperref}

\def\Rc{\mathbb{R}}

\newcommand{\F}{\mathcal{F}}
\newtheorem{thm}{Theorem}
\newtheorem{prop}{Proposition}

\newtheorem{definition}{Definition}

\newtheorem{remark}{Remark}

\begin{document}

\title{
Certain upper bounds on the eigenvalues 
associated with prolate spheroidal wave functions}
\author{Andrei Osipov\footnote{This author's research was supported in part
 by the AFOSR grant \#FA9550-09-1-0241}
\footnote{Yale University, 51 Prospect st, New Haven, CT 06511.
Email: andrei.osipov@yale.edu.
}}
\maketitle

\begin{abstract}
Prolate spheroidal wave functions (PSWFs) 
play an important role in various areas,
from physics (e.g. wave phenomena, fluid dynamics) to 
engineering (e.g. signal processing, filter design). 
One of the principal reasons for the importance of PSWFs
is that they are a natural and efficient tool for
computing with bandlimited functions, that
frequently occur in the abovementioned areas.
This is due to the fact that PSWFs are
the eigenfunctions of 
the integral operator, that represents timelimiting followed by
lowpassing. 
Needless to say, the behavior of this operator is governed by
the decay rate of its eigenvalues. Therefore,
investigation of this decay rate plays a crucial role
in the related theory and applications - for example,
in construction of quadratures, interpolation,
filter design, etc.

The significance
of PSWFs and, in particular, of the decay rate 
of the eigenvalues of the associated integral operator,
was realized at least half a century ago.
Nevertheless, perhaps surprisingly, despite vast numerical experience
and existence of several asymptotic expansions, a non-trivial
explicit upper bound on the magnitude of the eigenvalues
has been missing for decades.

The principal goal of this paper is
to close this gap in the theory of PSWFs.
We analyze the integral operator associated with PSWFs,
to derive fairly tight non-asymptotic
upper bounds on the magnitude of its
eigenvalues.
Our results are illustrated via several numerical experiments.
\end{abstract}


\noindent
{\bf Math subject classification:} {
33E10, 34L15, 35S30, 42C10, 45C05, 54P05}

\section{Introduction}
\label{sec_intro}

The principal purpose of this paper is to establish and prove
several inequalities involving 
the eigenvalues of a certain integral operator
associated with bandlimited functions
(see Section~\ref{sec_summary} below). While some of these inequalities
are known from ``numerical experience''
(see, for example,
\cite{RokhlinXiaoApprox},
\cite{ProlateLandau1},
\cite{SlepianLambdas}),
their proofs appear to be absent in the literature.

A function $f: \Rc \to \Rc$ is bandlimited of band limit $c>0$, if there
exists a function $\sigma \in L^2\left[-1,1\right]$ such that
\begin{align}
f(x) = \int_{-1}^1 \sigma(t) e^{icxt} \; dt.
\label{eq_intro_f}
\end{align}
In other words, the Fourier transform of a bandlimited function
is compactly supported.
While \eqref{eq_intro_f} defines $f$ for all real $x$, 
one is often interested in bandlimited functions, whose 
argument is confined to an interval, e.g. $-1 \leq x \leq 1$.
Such functions are encountered in physics (wave phenomena,
fluid dynamics), engineering (signal processing), etc.
(see e.g. \cite{SlepianComments}, \cite{Flammer}, \cite{Papoulis}).

About 50 years ago it was observed that the eigenfunctions of
the integral operator $F_c: L^2\left[-1,1\right] \to L^2\left[-1,1\right]$,
defined via the formula
\begin{align}
F_c\left[\varphi\right] \left(x\right) = \int_{-1}^1 \varphi(t) e^{icxt} \; dt,
\label{eq_intro_fc}
\end{align}
provide a natural tool for dealing with bandlimited functions, defined
on the interval $\left[-1,1\right]$. Moreover, it
was observed 
(see \cite{ProlateSlepian1}, \cite{ProlateLandau1}, \cite{ProlateSlepian2})
that the eigenfunctions of $F_c$
are precisely the prolate spheroidal wave functions (PSWFs),
well known from the mathematical physics (see, for example,
\cite{PhysicsMorse}, \cite{Flammer}).
The PSWFs are the eigenfunctions
of the differential operator $L_c$, defined via the formula
\begin{align}
L_c\left[ \varphi \right] \left(x\right)= 
-\frac{d}{dx} \left( (1-x^2) \cdot \frac{d\varphi}{dx}(x) \right) +
c^2 x^2.
\label{eq_intro_lc}
\end{align}
In other words, the integral operator $F_c$ 
commutes with
the differential
operator $L_c$ (see
\cite{ProlateSlepian1}, \cite{Grunbaum}).
This property, being remarkable by itself,
also plays an important role in both the analysis of PSWFs
and the associated numerical algorithms (see, for example,
\cite{Glaser}, \cite{RokhlinXiaoProlate}).

Obviously, the behavior of the operator $F_c$ is governed
by the decay rate of its eigenvalues.
Over the last half a century, several related asymptotic expansions,
as well as results of numerous numerical experiments,
have been published; moreover, 
implications of the decay rate of the eigenvalues to both theory
and applications have been extensively covered
in the literature
- see, for example,
\cite{Yoel}, 
\cite{RokhlinXiaoProlate}, 
\cite{RokhlinXiaoApprox}.
\cite{RokhlinXiaoAsymptotic},
\cite{LandauWidom},
\cite{ProlateSlepian1},
\cite{ProlateLandau1},
\cite{ProlateLandau2},
\cite{ProlateSlepian2},
\cite{ProlateSlepian3},
\cite{SlepianAsymptotic},
\cite{SlepianLambdas},
\cite{Fuchs}.
It is perhaps surprising, however, that a non-trivial
explicit
upper bound on the magnitude of the eigenvalues of $F_c$
has been missing for decades.
This paper closes this gap in the theory of PSWFs.

This paper is mostly devoted to the analysis
of the integral operator $F_c$, defined via \eqref{eq_intro_fc}.
More specifically, several explicit upper bounds for the 
magnitude of the eigenvalues of $F_c$
are derived. These bounds turn out to be fairly tight.
The analysis is illustrated through several numerical experiments.

Some of the results of this paper are based on the recent
analysis of the differential operator $L_c$,
defined via \eqref{eq_intro_lc}, that appears in \cite{Report}, 
\cite{ReportArxiv}.
Nevertheless, the techniques used in this paper are quite
different from those of \cite{Report}, \cite{ReportArxiv}.
The implications of the recent analysis of both $L_c$ and $F_c$
to numerical algorithms involving
PSWFs are being currently investigated.

This paper is organized as follows. In Section~\ref{sec_prel},
we summarize a number of well known mathematical facts
to be used in the rest of this paper. In Section~\ref{sec_summary},
we provide a summary of the principal results
of this paper, and discuss several consequences of these results. 
In Section~\ref{sec_analytical}, we introduce
the necessary analytical apparatus and carry out the analysis.
In Section~\ref{sec_numerical}, we illustrate the analysis
via several numerical examples.

\section{Mathematical and Numerical Preliminaries}
\label{sec_prel}
In this section, we introduce notation and summarize
several facts to be used in the rest of the paper.

\subsection{Prolate Spheroidal Wave Functions}
\label{sec_pswf}

In this subsection, we summarize several facts about
the PSWFs. Unless stated otherwise, all of these facts can be 
found in \cite{RokhlinXiaoProlate}, 
\cite{RokhlinXiaoApprox},
\cite{LandauWidom},
\cite{ProlateSlepian1},
\cite{ProlateLandau1},
\cite{Report},
\cite{ReportArxiv}.

Given a real number $c > 0$, we define the operator
$F_c: L^2\left[-1, 1\right] \to L^2\left[-1, 1\right]$ via the formula
\begin{align}
F_c\left[\varphi\right] \left(x\right) = \int_{-1}^1 \varphi(t) e^{icxt} \; dt.
\label{eq_pswf_fc}
\end{align}
Obviously, $F_c$ is compact. We denote its eigenvalues by
$\lambda_0, \lambda_1, \dots, \lambda_n, \dots$ and assume that
they are ordered such that 
$\left|\lambda_n\right| \geq \left|\lambda_{n+1}\right|$
for all natural $n \geq 0$. We denote by $\psi_n$ the eigenfunction
corresponding to $\lambda_n$. In other words, the following
identity holds for all integer $n \geq 0$ and all real $-1 \leq x \leq 1$:
\begin{align}
\label{eq_prolate_integral}
\lambda_n \psi_n\left(x\right) = \int_{-1}^1 \psi_n(t) e^{icxt} \; dt.
\end{align}
We adopt the convention\footnote{
This convention agrees with that of \cite{RokhlinXiaoProlate},
\cite{RokhlinXiaoApprox} and differs from that of \cite{ProlateSlepian1}.
}
that $\| \psi_n \|_{L^2\left[-1,1\right]} = 1$.
The following theorem describes the eigenvalues and eigenfunctions
of $F_c$
(see
\cite{RokhlinXiaoProlate},
\cite{RokhlinXiaoApprox},
\cite{ProlateSlepian1}).
\begin{thm}
Suppose that $c>0$ is a real number, and that the operator $F_c$
is defined via \eqref{eq_pswf_fc} above. Then,
the eigenfunctions $\psi_0, \psi_1, \dots$ of $F_c$ are purely real,
are orthonormal and are complete in $L^2\left[-1, 1\right]$.
The even-numbered functions are even, the odd-numbered ones are odd.
Each function $\psi_n$ has exactly $n$ simple roots in $\left(-1, 1\right)$.
All eigenvalues $\lambda_n$ of $F_c$ are non-zero and simple;
the even-numbered ones are purely real and the odd-numbered ones
are purely imaginary; in particular, $\lambda_n = i^n \left|\lambda_n\right|$.
\label{thm_pswf_main}
\end{thm}
\noindent
We define the self-adjoint operator
$Q_c: L^2\left[-1, 1\right] \to L^2\left[-1, 1\right]$ via the formula
\begin{align}
Q_c\left[\varphi\right] \left(x\right) =
\frac{1}{\pi} \int_{-1}^1 
\frac{ \sin \left(c\left(x-t\right)\right) }{x-t} \; \varphi(t) \; dt.
\label{eq_pswf_qc}
\end{align}
Clearly, if we denote by $\F:L^2(\Rc) \to L^2(\Rc)$ 
the unitary Fourier transform,
then
\begin{align}
Q_c\left[\varphi\right] \left(x\right) = 
\chi_{\left[-1,1\right]}(x) \cdot
\F^{-1} \left[ 
  \chi_{\left[-c,c\right]}(\xi) \cdot 
\F\left[\varphi\right](\xi)
\right](x),
\end{align}
where $\chi_{\left[-a,a\right]} : \Rc \to \Rc$ is the characteristic
function of the interval $\left[-a,a\right]$, defined via the formula
\begin{align}
\chi_{\left[-a,a\right]}(x) = 
\begin{cases}
1 & -a \leq x \leq a, \\
0 & \text{otherwise},
\end{cases}
\label{eq_char_function}
\end{align}
for all real $x$.
In other words, $Q_c$ represents low-passing followed by time-limiting.
$Q_c$ relates to $F_c$, defined via \eqref{eq_pswf_fc}, by 
\begin{align}
Q_c = \frac{ c }{ 2 \pi } \cdot F_c^{\ast} \cdot F_c,
\label{eq_pswf_qc_fc}
\end{align}
and the eigenvalues $\mu_n$ of $Q_n$ satisfy the identity
\begin{align}
\mu_n = \frac{c}{2\pi} \cdot \left|\lambda_n\right|^2,
\label{eq_prolate_mu}
\end{align}
for all integer $n \geq 0$.
Moreover, $Q_c$ has the same eigenfunctions $\psi_n$ as $F_c$.
In other words,
\begin{align}
\mu_n \psi_n(x) = \frac{1}{\pi} 
      \int_{-1}^1 \frac{ \sin\left(c\left(x-t\right)\right) }
                       { x - t } \; \psi_n(t) \; dt,
\label{eq_prolate_integral2}
\end{align}
for all integer $n \geq 0$ and all $-1 \leq x \leq 1$.
Also,  $Q_c$ is closely related to the operator
$P_c: L^2(\Rc) \to L^2(\Rc)$,
defined via the formula
\begin{align}
P_c\left[\varphi\right] \left(x\right) =
\frac{1}{\pi} \int_{-\infty}^{\infty}
\frac{ \sin \left(c\left(x-t\right)\right) }{x-t} \; \varphi(t) \; dt,
\label{eq_pswf_pc}
\end{align}
which is a widely known orthogonal projection onto the space
of functions of band limit $c > 0$ on the real
line $\Rc$.

The following theorem about the eigenvalues $\mu_n$ of the operator $Q_c$,
defined via
\eqref{eq_pswf_qc},
can be traced back to \cite{LandauWidom}:
\begin{thm}
Suppose that $c>0$ and $0<\alpha<1$ are positive real numbers,
and that the operator $Q_c: L^2\left[-1,1\right] \to L^2\left[-1,1\right]$
is defined via \eqref{eq_pswf_qc} above.
Suppose also that the integer $N(c,\alpha)$ is the number of 
the eigenvalues $\mu_n$ of $Q_c$ that are greater than $\alpha$. In
other words,
\begin{align}
N(c,\alpha) = \max\left\{ k = 1,2,\dots \; : \; \mu_{k-1} > \alpha\right\}.
\end{align}
Then,
\begin{align}
N(c,\alpha)
= \frac{2c}{\pi} + \left( \frac{1}{\pi^2} \log \frac{1-\alpha}{\alpha} \right)
    \log c + O\left( \log c \right).
\label{eq_mu_spectrum}
\end{align}
\label{thm_mu_spectrum}
\end{thm}
\noindent
According to \eqref{eq_mu_spectrum}, there are about $2c/\pi$
eigenvalues whose absolute value is close to one, order of $\log c$
eigenvalues that decay exponentially, and the rest of them are
very close to zero. 

The eigenfunctions $\psi_n$ of $Q_c$ turn out to be the PSWFs, well
known from classical mathematical physics \cite{PhysicsMorse}.
The following theorem, proved in a more general form in
\cite{ProlateSlepian2},
formalizes this statement.
\begin{thm}
For any $c > 0$, there exists a strictly increasing unbounded sequence
of positive numbers $\chi_0 <  \chi_1 <  \dots$ such that, for
each integer $n \geq 0$, the differential equation
\begin{align}
\left(1 - x^2\right) \cdot \psi''(x) - 2 x \cdot \psi'(x) 
+ \left(\chi_n - c^2 x^2\right) \cdot \psi(x) = 0
\label{eq_prolate_ode}
\end{align}
has a solution that is continuous on $\left[-1, 1\right]$.
Moreover, all such solutions are constant multiples of 
the eigenfunction $\psi_n$ of $F_c$,
defined via \eqref{eq_pswf_fc} above.
\label{thm_prolate_ode}
\end{thm}

In the following theorem, that appears in \cite{RokhlinXiaoApprox},
an upper bound on $\left|\lambda_n\right|$ in terms of $n$ and $c$
is described 
(the accuracy of this bound is
discussed in Section~\ref{sec_inac} below;
see also Theorem~\ref{thm_nu} and Remark~\ref{rem_nu}
 in Section~\ref{sec_weaker}).
\begin{thm}
Suppose that $c>0$ is a real number, and $n \geq 0$ is
a non-negative integer. Suppose also that $\lambda_n$
is the $n$th eigenvalue of the operator $F_c$, defined
via \eqref{eq_pswf_fc}. Suppose furthermore that the real
number $\nu(n,c)$ is defined via the formula
\begin{align}
\nu(n,c) = \frac{ \sqrt{\pi} \cdot c^n \left(n!\right)^2 }
                   { \left(2n\right)! \cdot \Gamma(n+3/2) },
\label{eq_nu}
\end{align}
where $\Gamma$ denotes the gamma function. Then,
\begin{align}
\left| \lambda_n \right| \leq \nu(n,c).
\label{eq_lambda_nu}
\end{align}
Moreover,
\begin{align}
\lambda_n(c) = i^n \nu(n,c) \cdot e^{R(n, c)}, 
\label{eq_lambda_formula}
\end{align}
where the real number $R(n,c)$ is defined via the formula
\begin{align}
R(n, c) = \int_0^c \left( 
   \frac{2 \left(\psi_n^{\tau}(1)\right)^2 -  1 }{ 2 \tau}
        - \frac{n}{\tau} \right) d\tau.
\label{eq_rnc}
\end{align}
The function $\psi_n^\tau$ in \eqref{eq_rnc} is the $n$th PSWF corresponding
to the band limit $\tau$.
\label{thm_rokhlin}
\end{thm}
\noindent
The following approximation formula for $|\lambda_n|$ appears
in Theorem 18 of \cite{RokhlinXiaoApprox}, without proof
(though the authors do illustrate its accuracy
via several numerical examples).
\begin{thm}
Suppose that $c \geq 1$ is a real number, and that $n \geq c$ is
a positive integer. Suppose also that the real number 
$p_0(n,c)$ is defined via the formula
\begin{align}
p_0(n,c) =
\sqrt{\frac{2\pi}{c}} \cdot
\exp\left[
-\sqrt{\chi_n} \cdot \left(
F\left( \sqrt{\frac{\chi_n-c^2}{\chi_n}} \right) -
E\left( \sqrt{\frac{\chi_n-c^2}{\chi_n}} \right)
\right)
\right],
\label{eq_p0}
\end{align}
where $F,E$ are the complete elliptic integrals,
defined, respectively, via \eqref{eq_F}, \eqref{eq_E}
in Section~\ref{sec_elliptic}.
Then,
\begin{align}
\left| \frac{|\lambda_n|}{p_0(n,c)} - 1 \right| 
= O\left( \frac{1}{\sqrt{cn}} \right).
\label{eq_lambda_approx}
\end{align}
\label{thm_approx}
\end{thm}
\begin{remark}
Obviously, \eqref{eq_lambda_approx} cannot be used
in rigorous analysis, due to the lack of both error
estimates and proof. In addition, the assumption $n \geq c$ turns
out to be rather restrictive. Nevertheless, in Section~\ref{sec_analytical}
we establish several upper bounds on $|\lambda_n|$, whose 
form is similar to that of $p_0(n,c)$. The approximate
formula \eqref{eq_lambda_approx} will only be used
in the discussion of the accuracy of these bounds,
in Section~\ref{sec_inac}.
\label{rem_approx}
\end{remark}

\noindent
The following four theorems contain relatively recent results.
All of them appear in \cite{Report}, \cite{ReportArxiv}.

Many properties of the PSWF $\psi_n$ depend on
whether the eigenvalue $\chi_n$ of the ODE \eqref{eq_prolate_ode}
is greater than or less than $c^2$. 
In the following theorem from \cite{Report}, \cite{ReportArxiv}, we describe
a simple relationship 
between $c, n$ and $\chi_n$.
\begin{thm}
Suppose that $n \geq 2$ is a non-negative integer.
\begin{itemize}
\item If $n \leq (2c/\pi)-1$, then $\chi_n < c^2$.
\item If $n \geq (2c/\pi)$, then $\chi_n > c^2$.
\item If $(2c/\pi)-1 < n < (2c/\pi)$, then either inequality is possible.
\end{itemize}
\label{thm_n_and_khi}
\end{thm}
\noindent
In the following theorem from \cite{Report}, \cite{ReportArxiv},
we describe upper and lower bounds on $\chi_n$ in terms
of $n$ and $c$.
\begin{thm}
Suppose that $n \geq 2$ is a positive integer, and that $\chi_n > c^2$. Then,
\begin{align}
n < & \; 
\frac{2}{\pi} \int_0^1 \sqrt{ \frac{\chi_n - c^2 t^2}{1 - t^2} } \; dt
= \nonumber \\
& \; \frac{2}{\pi} \sqrt{\chi_n} \cdot E \left( \frac{c}{\sqrt{\chi_n}} \right)
< n+3,
\label{eq_both_large_simple_prop}
\end{align}
where the function $E:\left[0,1\right] \to \Rc$ is defined via
\eqref{eq_E} in Section~\ref{sec_elliptic}.
\label{thm_n_khi_simple}
\end{thm}
\noindent
In the following theorem, 
we provide another upper bound on $\chi_n$ in terms of $n$.
\begin{thm}
Suppose that $n \geq 2$ is a positive integer, and that $\chi_n > c^2$. Then,
\begin{align}
\chi_n < \left( \frac{\pi}{2} \left(n+1\right) \right)^2.
\label{eq_khi_n_square_prop}
\end{align}
\label{thm_khi_n_square}
\end{thm}
\noindent
In the following theorem, we describe an upper bound on
the reciprocal of $\left|\psi_n(0)\right|$ for even $n$
(see Theorem 21 in \cite{Report}).
\begin{thm}
Suppose that $n > 0$ is an even integer,
and that $\chi_n > c^2$. Then,
\begin{align}
\frac{1}{| \psi_n(0) |} \leq 4 \cdot \sqrt{n \cdot \frac{\chi_n}{c^2}}.
\label{eq_prop_psi0}
\end{align}
\label{thm_psi0}
\end{thm}
\begin{remark}
Detailed numerical experiments, conducted by the author,
seem to indicate that, in fact,
\begin{align}
\frac{1}{| \psi_n(0) |} = O(1)
\label{eq_psi0_approx}
\end{align}
(see also \cite{RokhlinXiaoApprox}).
In other words, the inequality \eqref{eq_prop_psi0} is rather crude;
on the other hands, it has been rigorously proved, and
is sufficient for our purposes.
\label{rem_psi0}
\end{remark}

\subsection{Legendre Polynomials and PSWFs}
\label{sec_legendre}
In this subsection, we list several well known
facts about Legendre polynomials and the relationship between
Legendre polynomials and PSWFs.
All of these facts can be found, for example, in 
\cite{Ryzhik},
\cite{RokhlinXiaoProlate}
\cite{Abramovitz}.

The Legendre polynomials $P_0, P_1, P_2, \dots$ are defined via
the formulae
\begin{align}
& P_0(t) = 1, \nonumber \\
& P_1(t) = t,
\label{eq_legendre_pol_0_1}
\end{align}
and the recurrence relation
\begin{align}
\left(k+1\right) P_{k+1}(t) = \left(2k+1\right) t P_k(t) - k P_{k-1}(t),
\label{eq_legendre_pol_rec}
\end{align}
for all $k = 1, 2, \dots$.
The Legendre polynomials $\left\{P_k\right\}_{k=0}^{\infty}$ constitute a 
complete orthogonal system in $L^2\left[-1, 1\right]$. The normalized
Legendre polynomials are defined via the formula
\begin{align}
\overline{P_k}(t) = P_k(t) \cdot \sqrt{k + 1/2}, 
\label{eq_legendre_normalized}
\end{align}
for all $k=0,1,2,\dots$. The $L^2\left[-1,1\right]$-norm of each
normalized Legendre polynomial equals to one, i.e.
\begin{align}
\int_{-1}^1 \left( \overline{P_k}(t) \right)^2 \; dt = 1.
\label{eq_legendre_normalized_norm}
\end{align}
Therefore, the normalized Legendre polynomials constitute
an orthonormal basis for $L^2\left[-1, 1\right]$. 
In particular, for every real $c>0$
and every integer $n \geq 0$, the prolate spheroidal
wave function $\psi_n$, corresponding
to the band limit $c$, can be expanded into the series
\begin{align}
\psi_n(x) = \sum_{k = 0}^{\infty} \beta_k^{(n,c)} \cdot \overline{P_k}(x),
\label{eq_num_leg_exp}
\end{align}
for all $-1 \leq x \leq 1$,
where $\beta_0^{(n,c)}, \beta_1^{(n,c)}, \dots$ are defined via
the formula
\begin{align}
\beta_k^{(n,c)} = \int_{-1}^1 \psi_n(x) \cdot \overline{P_k}(x) \; dx,
\end{align}
for all $k=0, 1, 2, \dots$.
The sequence $\beta_0^{(n,c)}, \beta_1^{(n,c)}, \dots$ satisfies
the recurrence relation
\begin{align}
A_{0,0} \cdot \beta_0^{(n,c)} + 
A_{0,2} \cdot \beta_2^{(n,c)} & = \chi_n \cdot \beta_0^{(n,c)}, 
\nonumber \\
A_{1,1} \cdot \beta_1^{(n,c)} + 
A_{1,3} \cdot \beta_3^{(n,c)} & = \chi_n \cdot \beta_1^{(n,c)}, 
\nonumber \\
A_{k,k-2} \cdot \beta_{k-2}^{(n,c)} + 
A_{k,k} \cdot \beta_k^{(n,c)} +
A_{k,k+2} \cdot \beta_{k+2}^{(n,c)} & = \chi_n \cdot \beta_k^{(n,c)}, 
\end{align}
for all $k=2,3,\dots$, where $A_{k,k}$, $A_{k+2,k}$, $A_{k,k+2}$ are
defined via the formulae
\begin{align}
& A_{k, k} = k(k+1) + \frac{ 2k(k+1) - 1 }{ (2k+3)(2k-1) } \cdot c^2, 
  \nonumber \\
& A_{k, k+2} = A_{k+2, k} =
\frac{ (k+2)(k+1) }{ (2k+3) \sqrt{(2k+1)(2k+5)} } \cdot c^2,
\label{eq_num_a_matrix}
\end{align}
for all $k=0,1,2,\dots$.
In other words,
the infinite vector $\beta = \left\{\beta_k^{(n,c)}\right\}_{k = 0}^{\infty}$
satisfies the identity
\begin{align}
\left(A - \chi_n I\right) \cdot \beta = 0,
\label{eq_num_beta}
\end{align}
where the non-zero entries of the infinite symmetric matrix $A$ are given via
\eqref{eq_num_a_matrix}. 

\subsection{Elliptic Integrals}
\label{sec_elliptic}
In this subsection, we summarize several facts about
elliptic integrals. These facts can be found, for example,
in section 8.1 in \cite{Ryzhik}, and in \cite{Abramovitz}.

The incomplete elliptic integrals of the first and second kind
are defined, respectively, by the formulae
\begin{align}
\label{eq_F_y}
& F(y, k) =  \int_0^y \frac{dt}{\sqrt{1 - k^2 \sin^2 t}}, \\
& E(y, k) = \int_0^y \sqrt{1 - k^2 \sin^2 t} \; dt,
\label{eq_E_y}
\end{align}
where $0 \leq y \leq \pi/2$ and $0 \leq k \leq 1$.
By performing the substitution $x = \sin t$, we can write 
\eqref{eq_F_y} and \eqref{eq_E_y} as
\begin{align}
& F(y, k) = \int_0^{\sin(y)}
  \frac{ dx }{ \sqrt{\left(1 - x^2\right) \left(1 - k^2 x^2\right) } },
\label{eq_F_y_2} \\
\nonumber \\
& E(y, k) = \int_0^{\sin(y)}
\sqrt{ \frac{1 - k^2 x^2}{1 - x^2} } \; dx.
\label{eq_E_y_2}
\end{align}
The complete elliptic integrals of the first and second kind are
defined, respectively, by the formulae
\begin{align}
\label{eq_F}
& F(k) = F\left(\frac{\pi}{2}, k\right) = 
\int_0^{\pi/2} \frac{dt}{\sqrt{1 - k^2 \sin^2 t}}, \\
& E(k) = E\left(\frac{\pi}{2}, k\right) =
\int_0^{\pi/2} \sqrt{1 - k^2 \sin^2 t} \; dt,
\label{eq_E}
\end{align}
for all $0 \leq k \leq 1$. Moreover,
\begin{align}
E\left( \sqrt{1-k^2} \right) =
1 + \left(-\frac{1}{4}+\log(2)-\frac{\log(k)}{2}\right) \cdot k^2 +
    O\left( k^4 \cdot \log(k) \right).
\label{eq_E_exp}
\end{align}
In addition,
\begin{align}
F(k)-E(k) > \frac{\pi}{4} \cdot k^2,
\label{eq_F_minus_E}
\end{align}
for all real $0 < k < 1$.

\section{Summary and Discussion}
\label{sec_summary}
In this section, we summarize some of the properties of 
prolate spheroidal wave functions and the associated eigenvalues,
proved in Section~\ref{sec_analytical}. 
In particular, we present several upper bounds on $|\lambda_n|$
and discuss their accuracy.
The PSWFs and related notions were introduced in Section~\ref{sec_pswf}.
Throughout this section, the band limit $c>0$ is assumed
to be a positive real number.

\subsection{Summary of Analysis}
\label{sec_sum_analysis}
In the following two propositions, we provide some upper bounds
on the eigenvalues $\chi_n$ of the ODE \eqref{eq_prolate_ode}.
They are proved in Theorem~\ref{thm_exp_term}, 
\ref{thm_big_h}, \ref{thm_khi_n_upper}
in Section~\ref{sec_weaker}.
\begin{prop}
Suppose that $n$ is a positive integer, and that
\begin{align}
n > \frac{2c}{\pi} + \frac{2}{\pi^2} \cdot \delta
   \cdot \log \left( \frac{4e\pi c}{\delta} \right),
\end{align}
for some
\begin{align}
0 < \delta < \frac{5\pi}{4} \cdot c.
\end{align}
Then,
\begin{align}
\chi_n > c^2 + \frac{4}{\pi} \cdot \delta \cdot c.
\end{align}
\label{prop_khi_1}
\end{prop}
\begin{prop}
Suppose that $n$ is a positive integer, and that
\begin{align}
\frac{2c}{\pi} \leq 
n \leq \frac{2c}{\pi} + \frac{2}{\pi^2} \cdot \delta
   \cdot \log \left( \frac{4e\pi c}{\delta} \right) - 3,
\label{eq_n_greater}
\end{align}
for some
\begin{align}
3 < \delta < \frac{5\pi}{4} \cdot c.
\label{eq_delta_khi_2}
\end{align}
Then,
\begin{align}
\chi_n < c^2 + \frac{8 \cdot \delta}{\pi} \cdot c.
\end{align}
\label{prop_khi_2}
\end{prop}
The following is one of the principal results of this paper.
It is proved in Theorem~\ref{thm_big_inequality}
in Section~\ref{sec_principal} (see also Remark~\ref{rmk_42}),
and is illustrated in Experiments 2, 3 in Section~\ref{sec_numerical}.
\begin{prop}
Suppose that $n>0$ is an even integer number, 
and that $\lambda_n$ is the $n$th eigenvalue of the
integral operator $F_c$,
defined via \eqref{eq_pswf_fc}, \eqref{eq_prolate_integral}
in Section~\ref{sec_pswf}. 
Suppose also that
\begin{align}
n > \frac{2c}{\pi} + \sqrt{42}.
\end{align}
Suppose furthermore that the real number $\zeta(n,c)$ is defined
via the formula
\begin{align}
\zeta(n,c) = & \;
\frac{7}{2|\psi_n(0)|} \cdot
\frac{ \left(4 \cdot \chi_n/c^2- 2 \right)^{4} }
     {       3 \cdot \chi_n/c^2 - 1            }  \cdot
\left(\chi_n-c^2\right)^{\frac{1}{4}} \cdot \nonumber \\
& \; \exp\left[
-\sqrt{\chi_n} \cdot \left(
F\left( \sqrt{\frac{\chi_n-c^2}{\chi_n}} \right) -
E\left( \sqrt{\frac{\chi_n-c^2}{\chi_n}} \right)
\right)
\right],
\label{eq_zeta_n_c_prop}
\end{align}
where $\chi_n$ is the $n$th eigenvalue of the differential operator $L_c$,
defined via \eqref{eq_intro_lc} in Section~\ref{sec_intro}, and
$F,E$ are the complete elliptic integrals,
defined, respectively, via \eqref{eq_F}, \eqref{eq_E}
in Section~\ref{sec_elliptic}.
Then,
\begin{align}
| \lambda_n | < \zeta(n,c).
\label{eq_big_inequality_prop}
\end{align}
\label{prop_big_inequality}
\end{prop}
\begin{remark}
It follows from the combination of
Remark~\ref{rem_psi0} in Section~\ref{sec_pswf}
and Proposition~\ref{prop_khi_2} above that
\begin{align}
\zeta(n,c) = 
 O((\delta c)^{1/4}) \cdot
\exp\left[
-\sqrt{\chi_n} \cdot \left(
F\left( \sqrt{\frac{\chi_n-c^2}{\chi_n}} \right) -
E\left( \sqrt{\frac{\chi_n-c^2}{\chi_n}} \right)
\right)
\right],
\label{eq_zeta_approx}
\end{align}
where $n, \delta$ are as in \eqref{eq_n_greater},
\eqref{eq_delta_khi_2}.
\label{rem_big}
\end{remark}
In the following proposition, we describe 
another upper
bound on $| \lambda_n |$, which is weaker than the one presented
in Proposition~\ref{prop_big_inequality}, but has a simpler form.
It is proved in Theorem~\ref{thm_simple_inequality}
in Section~\ref{sec_weaker}.
\begin{prop}
Suppose that $n>0$ is an even integer number, 
and that $\lambda_n$ is the $n$th eigenvalue of the
integral operator $F_c$,
defined via \eqref{eq_pswf_fc}, \eqref{eq_prolate_integral}
in Section~\ref{sec_pswf}. 
Suppose also that
\begin{align}
n > \frac{2c}{\pi} + \sqrt{42}.
\end{align}
Suppose furthermore that
the real number $\eta(n,c)$ is defined via the formula
\begin{align}
& \eta(n,c) = \nonumber \\
& 18 \cdot (n+1) \cdot \left( \frac{\pi \cdot (n+1)}{c} \right)^7 \cdot
 \exp\left[
-\sqrt{\chi_n} \cdot \left(
F\left( \sqrt{\frac{\chi_n-c^2}{\chi_n}} \right) -
E\left( \sqrt{\frac{\chi_n-c^2}{\chi_n}} \right)
\right)
\right],
\label{eq_eta_n_c_prop}
\end{align}
where $\chi_n$ is the $n$th eigenvalue of the differential operator $L_c$,
defined via \eqref{eq_intro_lc} in Section~\ref{sec_intro}, and
$F,E$ are the complete elliptic integrals,
defined, respectively, via \eqref{eq_F}, \eqref{eq_E}
in Section~\ref{sec_elliptic}.
Then,
\begin{align}
| \lambda_n | < \eta(n,c).
\label{eq_simple_inequality_prop}
\end{align}
\label{prop_simple_inequality}
\end{prop}
\begin{remark}
According to Proposition~\ref{prop_simple_inequality}, 
\begin{align}
\eta(n,c) =
O(c) \cdot
\exp\left[
-\sqrt{\chi_n} \cdot \left(
F\left( \sqrt{\frac{\chi_n-c^2}{\chi_n}} \right) -
E\left( \sqrt{\frac{\chi_n-c^2}{\chi_n}} \right)
\right)
\right],
\label{eq_eta_approx}
\end{align}
as long as $n$ is proportional to $c$.
\label{rem_simple}
\end{remark}

Both $\zeta(n,c)$ and $\eta(n,c)$, defined, respectively,
via \eqref{eq_zeta_n_c_prop} in Proposition~\ref{prop_big_inequality}
and \eqref{eq_eta_n_c_prop} in Proposition~\ref{thm_simple_inequality},
depend on $\chi_n$, which 
somewhat obscures their behavior.
In the following proposition, 
we eliminate this inconvenience by providing yet another
upper bound on $|\lambda_n|$.
The simplicity of this bound, as well as the fact that it depends only on $n$ 
and $c$ (and not on $\chi_n$), make Proposition~\ref{prop_crude_inequality}
the principal result of this paper.

It is proved in Theorem~\ref{thm_crude_inequality} in Section~\ref{sec_weaker}
and is illustrated via Experiment 3 in Section~\ref{sec_numerical}.
\begin{prop}
Suppose that $c>0$ is a real number, and that
\begin{align}
c > 22.
\label{eq_c22}
\end{align}
Suppose also that $\delta>0$ is a real number, and that
\begin{align}
3 < \delta < \frac{\pi c}{16}.
\label{eq_delta_crude_prop}
\end{align}
Suppose, in addition,that $n$ is a positive integer, and that
\begin{align}
n \geq \frac{2}{\pi} c + \frac{2}{\pi^2} \cdot \delta \cdot
    \log\left( \frac{4e\pi c}{\delta} \right).
\label{eq_n_crude_prop}
\end{align}
Suppose furthermore that the real number $\xi(n,c)$ is defined
via the formula
\begin{align}
\xi(n,c) = 7056 \cdot c \cdot 
\exp\left[-\delta\left(1 - \frac{\delta}{2\pi c}\right) \right].
\label{eq_xi_n_c_prop}
\end{align}
Then, 
\begin{align}
| \lambda_n | < \xi(n,c).
\label{eq_crude_inequality_prop}
\end{align}
\label{prop_crude_inequality}
\end{prop}

\subsection{Accuracy of Upper Bounds on $|\lambda_n|$}
\label{sec_inac}
In this subsection, we discuss the accuracy of 
the upper bounds on $|\lambda_n|$, presented in
Propositions~\ref{prop_big_inequality},
~\ref{prop_simple_inequality},
~\ref{prop_crude_inequality}.
In this discussion, we use the analysis of Section~\ref{sec_analytical};
previously reported results; and numerous
numerical experiments, some of which are described
in Section~\ref{sec_numerical}. Throughout this subsection, we suppose that
$n$ is a positive integer in the range
\begin{align}
\frac{2c}{\pi} < n < \frac{2c}{\pi} + O(\log(c)).
\label{eq_n_inac}
\end{align}
According to the combination of
Theorem~\ref{thm_approx} in Section~\ref{sec_pswf} and Remark~\ref{rem_big}, 
\begin{align}
\frac{ \zeta(n,c) }{ |\lambda_n| } = O(c^{3/4}),
\label{eq_inac_1}
\end{align}
where $\zeta(n,c)$ is that
of Proposition~\ref{prop_big_inequality}. On the other hand,
both $|\lambda_n|$ and $\zeta(n,c)$ 
decay with $n$ roughly exponentially, at the same rate.
Thus, the inequality \eqref{eq_big_inequality_prop} in 
Proposition~\ref{prop_big_inequality} is reasonably tight
(see also Experiment 2, Experiment 3 in Section~\ref{sec_numerical}).

The factor $O(c^{3/4})$ in \eqref{eq_inac_1}
is an artifact of the analysis in Section~\ref{sec_expansion}.
The first source of inaccuracy is the inequality
\eqref{eq_ineq_inac} in the proof of Theorem~\ref{thm_alphak}.
In this inequality, $\left| a_k^{(n,c)} \right|$ bounded from above by $1$,
while numerical experiments indicate that
\begin{align}
\left| a_k^{(n,c)} \right| < O(c^{-1/2}),
\end{align}
for all integer $k > 0$.
This contributes to the factor
of order $c^{1/2}$ in \eqref{eq_inac_1}. The second source of inaccuracy
is Theorem~\ref{thm_rho}, which gives rise to the factor
\begin{align}
\frac{ \left(4 \cdot \chi_n/c^2- 2 \right)^{4} }
     {       3 \cdot \chi_n/c^2 - 1            }  \cdot
\left(\chi_n-c^2\right)^{\frac{1}{4}} = O(c^{1/4})
\label{eq_inac_khi}
\end{align}
in \eqref{eq_zeta_n_c_prop} (see also Proposition~\ref{prop_khi_2}).
This contributes
to another factor of order $c^{1/4}$ in \eqref{eq_inac_1}.

In Propositions~\ref{prop_simple_inequality},
\ref{prop_crude_inequality} we introduce two additional
upper bounds on $|\lambda_n|$,
namely, $\eta(n,c)$ and $\xi(n,c)$.
Due to Remarks~\ref{rem_big}, \ref{rem_simple}
and Proposition~\ref{prop_crude_inequality},
\begin{align}
& \eta(n,c) = \zeta(n,c) \cdot O(c^{3/4}), \nonumber \\
& \xi(n,c) = \zeta(n,c) \cdot O(c^{3/4}).
\label{eq_inac_2}
\end{align}
Thus, \eqref{eq_zeta_n_c} is a tighter upper bound on $|\lambda_n|$
than both \eqref{eq_eta_n_c} and \eqref{eq_xi_n_c}.
This is not surprising, since, due to
Theorems~\ref{thm_simple_inequality}, \ref{thm_crude_inequality},
$\eta(n,c)$ and $\xi(n,c)$ can be viewed as simplified and
less accurate versions of $\zeta(n,c)$.
There are two sources of the discrepancy
\eqref{eq_inac_2}.
First, 
in the proofs of
Theorems~\ref{thm_simple_inequality}, \ref{thm_crude_inequality},
the term $\left(\chi_n-c^2\right)^{1/4}$ is bounded from above by
$O(c^{1/2})$, while, in fact, it is of order $c^{1/4}$
(see \eqref{eq_inac_khi} above). 
Additional factor of order $c^{1/2}$ in \eqref{eq_inac_2}
is due to  Theorem~\ref{thm_psi0}
and Remark~\ref{rem_psi0} in Section~\ref{sec_pswf}.
See also results of numerical experiments,
reported in Section~\ref{sec_numerical}.

Finally, we observe that the upper bound $\nu(n,c)$ on $|\lambda_n|$,
introduced in Theorem~\ref{thm_rokhlin} in Section~\ref{sec_pswf},
is useless for $n$ as in \eqref{eq_n_inac},
due to the combination of Theorem~\ref{thm_nu} and Remark~\ref{rem_nu}
in Section~\ref{sec_weaker}. On the other hand, $\nu(n,c)$ can be used
to understand the behavior of $|\lambda_n|$ as $n \to \infty$,
for a fixed $c>0$.

\section{Analytical Apparatus}
\label{sec_analytical}
The purpose of this section is to provide
the analytical apparatus to be used in the
rest of the paper. This principal results
of this section are 
Theorems~\ref{thm_big_inequality}, \ref{thm_simple_inequality}.

\subsection{Legendre Expansion}
\label{sec_expansion}
In this subsection, we analyze the Legendre expansion of PSWFs,
introduced in Section~\ref{sec_legendre}. This analysis
will be subsequently used in Section~\ref{sec_principal}
to prove the principal result of this paper.

The following theorem is a direct consequence of the results
outlined in Section~\ref{sec_pswf} and Section~\ref{sec_legendre}.
\begin{thm}
\label{thm_aknc}
Suppose that $c>0$ is a real number, and $n>0$ is an even positive integer.
Suppose also that the numbers $a_1^{(n,c)}, a_2^{(n,c)}, \dots$
are defined via the formula
\begin{align}
a_k^{(n,c)} = \int_{-1}^1 \psi_n(t) \cdot \overline{P_{2k-2}}(t) \; dt,
\label{eq_aknc_def}
\end{align}
for $k = 1,2,\dots$,
where $\psi_n$ is the $n$th PSWF corresponding to band limit $c$,
and $\overline{P_k}$ is the $k$th normalized Legendre polynomial.
Then, the sequence $\left\{a_k^{(n,c)}\right\}$ satisfies the recurrence
relation
\begin{align}
& c_1 \cdot a_2^{(n,c)} + b_1 \cdot a_1^{(n,c)} = 0, \nonumber \\
& c_{k+1} \cdot a_{k+2}^{(n,c)} +
  b_{k+1} \cdot a_{k+1}^{(n,c)} + c_k \cdot a_k^{(n,c)} = 0,
\label{eq_aknc_rec}
\end{align}
for $k \geq 1$,
where the numbers $c_1,c_2,\dots$ are defined via the formula
\begin{align}
c_k = \frac{2k \cdot (2k-1)}{(4k-1) \cdot \sqrt{ (4k-3) \cdot (4k+1) }} 
\cdot c^2,
\label{eq_ck_def}
\end{align}
for $k \geq 1$,
and the numbers $b_1,b_2,\dots$ are defined via the formula
\begin{align}
b_k = 2\cdot(k-1)\cdot(2k-1) +
  \frac{2 \cdot (2k-1) \cdot (2k-2) - 1}{(4k-1) \cdot (4k-5)} \cdot c^2 -
  \chi_n, 
\label{eq_bk_def}
\end{align}
for $k \geq 1$.
Here $\chi_n$ is the $n$th eigenvalue of the prolate differential equation
\eqref{eq_prolate_ode}. Moreover,
\begin{align}
\psi_n(t) = \sum_{k=1}^{\infty} a_k^{(n,c)} \cdot \overline{P_{2k-2}}(t),
\label{eq_aknc_expansion}
\end{align}
and
\begin{align}
\sum_{k=1}^{\infty} \left( a_k^{(n,c)} \right)^2 = 1.
\label{eq_aknc_sum_one}
\end{align}
\end{thm}
\begin{proof}
To establish \eqref{eq_aknc_rec} and \eqref{eq_aknc_expansion},
we combine \eqref{eq_num_leg_exp}, \eqref{eq_num_a_matrix}, 
\eqref{eq_num_beta} in Section~\ref{sec_legendre} with 
Theorem~\ref{thm_pswf_main} in Section~\ref{sec_pswf}.
The identity \eqref{eq_aknc_sum_one} follows from the fact
that the normalized Legendre polynomials constitute an
orthonormal basis for $L^2\left[-1,1\right]$.
\end{proof}

In the rest of the section, $c>0$ is a fixed real number,
and $n>0$ is an even positive integer.

The following theorem provides an upper bound on $\left|a_1^{(n,c)}\right|$
in terms of the elements of another sequence.
\begin{thm}
Suppose that the sequence $\alpha_1,\alpha_2,\dots$ is defined via
the formula
\begin{align}
\alpha_k = \frac{ a_k^{(n,c)} }{ a_1^{(n,c)} },
\label{eq_alphak_def}
\end{align}
for $k \geq 1$,
where $a_1^{(n,c)}, a_2^{(n,c)}, \dots$ are defined via \eqref{eq_aknc_def}
in Theorem~\ref{thm_aknc}.
Then, the sequence $\alpha_1,\alpha_2,\dots$ satisfies the recurrence
relation
\begin{align}
& \alpha_1 = 1, \nonumber \\
& \alpha_2 = B_0, \nonumber \\
& \alpha_{k+2} = B_k \cdot \alpha_{k+1} - A_k \cdot \alpha_k, 
\label{eq_alphak_rec}
\end{align}
for $k \geq 1$,
where the sequence $A_1,A_2,\dots$ is defined via the formula
\begin{align}
A_k = \frac{ k \cdot (2k-1) \cdot (4k+3) }
           { (k+1) \cdot (2k+1) \cdot (4k-1) } \cdot
    \sqrt{ \frac{4k+5}{4k-3} },
\label{eq_bigak_def}
\end{align}
for $k \geq 1$,
and the sequence $B_0,B_1,\dots$ is defined via the formula
\begin{align}
B_k = \; & \left( \frac{\chi_n-2k\cdot(2k+1)}{c^2} \right) \cdot
      \frac{ (4k+3)\cdot\sqrt{ (4k+1)\cdot(4k+5) } }{ (2k+1)\cdot(2k+2) }\; -
      \nonumber \\
& \frac{ (4k\cdot(2k+1)-1) \cdot \sqrt{ (4k+1)\cdot(4k+5) } }
       { (4k-1) \cdot (2k+1) \cdot (2k+2) },
\label{eq_bigbk_def}
\end{align}
for $k \geq 0$.
Moreover, for every $k=1,2,\dots$,
\begin{align}
\left| a_1^{(n,c)} \right| \leq \frac{ 1 }{ \left| \alpha_k \right| }.
\label{eq_a1_alphak_ineq}
\end{align}
\label{thm_alphak}
\end{thm}
\begin{proof}
Due to \eqref{eq_aknc_rec} in Theorem~\ref{thm_aknc}, the recurrence
relation \eqref{eq_alphak_rec} holds with $A_k,B_k$'s defined via
the formulae
\begin{align}
A_k = \frac{ c_k }{ c_{k+1} }, \quad B_k = - \frac{ b_{k+1} }{ c_{k+1} },
\end{align}
where $c_k,b_k$'s are defined, respectively, via 
\eqref{eq_ck_def} and \eqref{eq_bk_def}. We observe that
\begin{align}
\frac{1}{c_{k+1}} = \frac{ (4k+3) \cdot \sqrt{(4k+1)\cdot(4k+5)} }
                         { (2k+1) \cdot (2k+2) } \cdot \frac{1}{c^2}
\end{align}
and readily obtain both \eqref{eq_bigak_def} and \eqref{eq_bigbk_def}.
Next, due to \eqref{eq_aknc_sum_one} and \eqref{eq_alphak_def},
\begin{align}
1 \geq \left| a_k^{(n,c)} \right| = 
 \left| \frac{ a_k^{(n,c)} }{ a_1^{(n,c)} } \right| \cdot 
 \left| a_1^{(n,c)} \right|
 = \left| \alpha_k \right| \cdot \left| a_1^{(n,c)} \right|,
\label{eq_ineq_inac}
\end{align}
for all $k=1,2,\dots$,
which implies \eqref{eq_a1_alphak_ineq}.
\end{proof}
It is somewhat easier to analyze a rescaled version of the sequence
$\left\{\alpha_k\right\}$ defined via \eqref{eq_alphak_def}
in Theorem~\ref{thm_alphak}. This
observation is reflected in the following theorem.
\begin{thm}
Suppose that the sequence $\beta_1,\beta_2,\dots$ is defined via the formula
\begin{align}
\beta_k = \alpha_k \cdot \sqrt{\frac{2}{4k-3}},
\label{eq_betak_def}
\end{align}
for $k \geq 1$,
where $\alpha_1,\alpha_2,\dots$ are defined via \eqref{eq_alphak_def}
in Theorem~\ref{thm_alphak} above. Suppose also that the sequence
$B^{\chi}_0, B^{\chi}_1, \dots$ is defined via the formula
\begin{align}
B^{\chi}_k = 
\frac{(4k+1)\cdot(4k+3)}{(2k+1)\cdot(2k+2)} \cdot
\left[ \frac{\chi_n-c^2-2k\cdot(2k+1)}{c^2} \right],
\label{eq_bkhi_def}
\end{align}
for $k \geq 0$.
Then, the sequence $\beta_1,\beta_2,\dots$ satisfies the recurrence relation
\begin{align}
& \beta_1 = \sqrt{2}, \nonumber \\
& \beta_2 = \tilde{B}_0 \cdot \sqrt{2}, \nonumber \\
& \beta_{k+2} = \tilde{B}_k \cdot \beta_{k+1} - 
                \tilde{A}_k \cdot \beta_k,
\label{eq_betak_rec}
\end{align}
for $k \geq 1$,
where $\tilde{A}_0, \tilde{A}_1,\dots$ are defined via the formula
\begin{align}
\tilde{A}_k = \frac{ k \cdot (2k-1) \cdot (4k+3) }
                   { (k+1) \cdot (2k+1) \cdot(4k-1) },
\label{eq_tildeak_def}
\end{align}
for $k \geq 0$,
and $\tilde{B}_0, \tilde{B}_1, \dots$ are defined via the formula
\begin{align}
\tilde{B}_k = B^{\chi}_k + 1 + \tilde{A}_k,
\label{eq_tildebk_def}
\end{align}
for $k \geq 0$.
\label{thm_betak}
\end{thm}
\begin{proof}
Due to \eqref{eq_alphak_rec} and \eqref{eq_betak_def}, we have for 
all $k = 1,2,\dots$
\begin{align}
\beta_{k+2} & = \sqrt{\frac{2}{4k+5}} \cdot \alpha_{k+2}
              = \sqrt{\frac{2}{4k+5}} \cdot B_k \cdot \alpha_{k+1} -
                \sqrt{\frac{2}{4k+5}} \cdot A_k \cdot \alpha_k \nonumber \\
 & =  \sqrt{\frac{4k+1}{4k+5}} \cdot B_k  \cdot 
           \sqrt{\frac{2}{4k+1}} \cdot \alpha_{k+1} -
      \sqrt{\frac{4k-3}{4k+5}} \cdot A_k  \cdot
           \sqrt{\frac{2}{4k-3}} \cdot \alpha_k,
\end{align}
and hence the recurrence relation \eqref{eq_betak_rec} holds with
\begin{align}
\tilde{A}_k = \sqrt{ \frac{4k-3}{4k+5} } \cdot A_k, \quad
\tilde{B}_k = \sqrt{ \frac{4k+1}{4k+5} } \cdot B_k.
\label{eq_tildeab_ab}
\end{align}
It remains to compute $\tilde{A}_k$'s and $\tilde{B}_k$'s. First,
we observe that \eqref{eq_tildeak_def} follows immediately
from the combination of 
\eqref{eq_bigak_def} with \eqref{eq_tildeab_ab}. Second, we combine
\eqref{eq_bigbk_def} with \eqref{eq_tildeab_ab} to conclude that,
for $k=1,2,\dots$,
\begin{align}
\tilde{B}_k 
& = \left[ \frac{\chi_n-2k\cdot(2k+1)}{c^2} \right] \cdot
    \frac{ (4k+3)\cdot(4k+1) }{ (2k+1)\cdot(2k+2) } -
    \frac{(8k^2+4k-1)\cdot(4k+1)}{(4k-1)\cdot(2k+1)\cdot(2k+2)} \nonumber \\
& = \frac{(4k+1)\cdot(4k+3)}{(2k+1)\cdot(2k+2)} \cdot
\left[ \frac{\chi_n-c^2-2k\cdot(2k+1)}{c^2} \right] + \nonumber \\ 
& \quad \; \; \frac{ (4k+3)\cdot(4k+1)\cdot(4k-1)-(4k+1)\cdot(8k^2+4k-1) }
          { (4k-1)\cdot(2k+1)\cdot(2k+2) } \nonumber \\
& = \frac{(4k+1)\cdot(4k+3)}{(2k+1)\cdot(2k+2)} \cdot
\left[ \frac{\chi_n-c^2-2k\cdot(2k+1)}{c^2} \right] + 1 \; + \nonumber \\
& \; \; \; \; \frac{ (4k+3)\cdot(4k+1)\cdot(4k-1)-(4k+1)\cdot(8k^2+4k-1)-
                     (4k-1)\cdot(2k+1)\cdot(2k+2) }
          { (4k-1)\cdot(2k+1)\cdot(2k+2) } \nonumber \\
& = 
\frac{(4k+1)\cdot(4k+3)}{(2k+1)\cdot(2k+2)} \cdot
\left[ \frac{\chi_n-c^2-2k\cdot(2k+1)}{c^2} \right] + 1 + \tilde{A}_k,
\end{align}
which completes the proof.
\end{proof}
The following theorem, in which we
establish the monotonicity of both $\left\{\alpha_k\right\}$
and $\left\{\beta_k\right\}$ up to a certain value of $k$,
is a consequence
of Theorem~\ref{thm_betak}.
\begin{thm}
Suppose that $\chi_n > c^2$, and that $\beta_1, \beta_2, \dots$
are defined via \eqref{eq_betak_def} in Theorem~\ref{thm_betak}.
Suppose also that the integer $k_0$ is defined via the formula
\begin{align}
k_0 & \; = 
  \max_k \left\{ k = 1,2,\dots \; : \; 2k \cdot (2k+1) < \chi_n-c^2 \right\}
      \nonumber \\
    & \; = \max_k \left\{
        k = 1,2,\dots \; : \; k \leq 
        \frac{1}{2}\cdot\sqrt{\chi_n-c^2+\frac{1}{4}}-\frac{1}{4}
    \right\}.
\label{eq_k0_def}
\end{align}
Then, 
\begin{align}
\sqrt{2} = 
\beta_1 < \beta_2 < \dots < \beta_{k_0} < \beta_{k_0+1} < \beta_{k_0+2},
\label{eq_betak_mon}
\end{align}
and also,
\begin{align}
1 = 
\alpha_1 < \alpha_2 < \dots < \alpha_{k_0} < \alpha_{k_0+1} < \alpha_{k_0+2},
\label{eq_alphak_mon}
\end{align}
where the sequences $\left\{\alpha_k\right\}$ and $\left\{\beta_k\right\}$ 
are defined
via \eqref{eq_alphak_def} and \eqref{eq_betak_def}, respectively.
\label{thm_monotone}
\end{thm}
\begin{proof}
Due to \eqref{eq_tildebk_def} in Theorem~\ref{thm_betak}
and the assumption that $\chi_n > c^2$,
\begin{align}
\tilde{B}_0 = \frac{3}{2} \cdot \frac{\chi_n-c^2}{c^2} + 1 > 1.
\label{eq_tildeb0}
\end{align}
Therefore, due to \eqref{eq_betak_rec} in Theorem~\ref{thm_betak},
\begin{align}
\beta_2 = \tilde{B}_0 \cdot \beta_1 > \beta_1.
\end{align}
By induction, suppose that $1 \leq k \leq k_0$ and assume that 
$\beta_k < \beta_{k+1}$. 
We observe that $\tilde{A}_k, \tilde{B}_k > 0$, and
combine this observation with \eqref{eq_betak_rec},
\eqref{eq_tildeak_def},
\eqref{eq_tildebk_def} and \eqref{eq_k0_def} to conclude that
\begin{align}
\beta_{k+2} = \beta_{k+1} + \tilde{B}_k \cdot \beta_{k+1} +
           \tilde{A}_k \cdot \left( \beta_{k+1} - \beta_k\right) > \beta_{k+1},
\end{align}
which implies \eqref{eq_betak_mon}.
To establish \eqref{eq_alphak_mon}, we use \eqref{eq_betak_def} and observe
that
\begin{align}
\frac{ \alpha_{k+1} }{ \alpha_k} = 
\sqrt{\frac{4k+1}{4k-3}} \cdot \frac{ \beta_{k+1} }{ \beta_k } >
\sqrt{\frac{4k+1}{4k-3}} > 1,
\end{align}
for all $1 \leq k \leq k_0+1$.
\end{proof}
In the following theorem, we bound the sequence
$\beta_1, \beta_2, \dots$,
defined via \eqref{eq_betak_def} in Theorem~\ref{thm_betak},
by another sequence from below.
\begin{thm}
Suppose that $\chi_n > c^2$, 
and that the sequence $\rho_1, \rho_2, \dots$, is defined 
via the formula
\begin{align}
\rho_k = \frac{ (4k-6) \cdot (4k-4) \cdot (4k+7) }
              { (4k-2) \cdot (4k) \cdot (4k+3) },
\label{eq_rho_def}
\end{align}
for all $k=1,2,\dots$. Suppose also that the sequence 
$A^{new}_1, A^{new}_2, \dots$ is defined via the formula
\begin{align}
A^{new}_k = \tilde{A}_k \cdot \rho_k,
\label{eq_anew_def}
\end{align}
for all $k=1,2,\dots$, where $\tilde{A}_k$ is defined via
\eqref{eq_tildeak_def} in Theorem~\ref{thm_betak}.
Suppose furthermore that the sequence $\beta^{new}_1, \beta^{new}_2, \dots$
is defined via the formulae
\begin{align}
& \beta^{new}_1 = \beta_1, \nonumber \\
& \beta^{new}_2 = \beta_2, \nonumber \\
& \beta^{new}_3 = \beta_3, \nonumber, \\
& \beta^{new}_{k+2} = 
(B^{\chi}_k + 1) \cdot \beta^{new}_{k+1} + 
A^{new}_k \cdot (\beta^{new}_{k+1} - \beta^{new}_k), 
\label{eq_betanew_def}
\end{align}
for $k \geq 2$,
where $\beta_1, \beta_2, \dots$
are defined via \eqref{eq_betak_def},
and $B^{\chi}_k$ is defined via \eqref{eq_bkhi_def}
in Theorem~\ref{thm_betak}. Then,
\begin{align}
A^{new}_k = \frac{4k-4}{4k+4} \cdot \frac{4k-6}{4k+2} \cdot 
            \frac{4k+7}{4k-1},
\label{eq_anew_for}
\end{align}
for all $k=0,1,\dots$, and also
\begin{align}
0 = A^{new}_1 < A^{new}_2 < A^{new}_3 < \dots < A^{new}_k < \dots < 1.
\label{eq_anew_mon}
\end{align}
Moreover, 
\begin{align}
\sqrt{2} = 
\beta^{new}_1 < \beta^{new}_2 < 
\dots < \beta^{new}_{k_0} < \beta^{new}_{k_0+1} < \beta^{new}_{k_0+2},
\label{eq_betanew_mon}
\end{align}
where $k_0$ is defined via 
\eqref{eq_k0_def} in Theorem~\ref{thm_monotone}.
In addition, 
\begin{align}
\beta^{new}_1 \leq \beta_1, \quad
\beta^{new}_2 \leq \beta_2, \quad \dots, \quad
\beta^{new}_{k_0+1} \leq \beta_{k_0+1}, \quad
\beta^{new}_{k_0+2} \leq \beta_{k_0+2}.
\label{eq_betanew_and_beta}
\end{align}
\label{thm_rho}
\end{thm}
\begin{proof}
The identity
\eqref{eq_anew_for} follows immediately from
the combination of \eqref{eq_tildeak_def} and
\eqref{eq_rho_def}.
The monotonicity of $\left\{A^{new}_k\right\}$ follows from the fact
that, if we view $A_k$ as a function of the real argument $k$,
\begin{align}
\frac{dA_k}{dk} =
\frac{
\left( \left( (3+k)\cdot 8k-19 \right) \cdot 2k - 51 \right) \cdot 8k+2 }
{(4k-1)^2 \cdot (k+1)^2 \cdot (2k+1)^2},
\end{align}
which is positive for all $k \geq 2$; combining this observation
with the fact that $A^{new}_k$ tends to 1 as $k \to \infty$,
we obtain \eqref{eq_anew_mon}. 

It follows from \eqref{eq_betanew_def}
by induction that $\beta^{new}_{j+2} > \beta^{new}_{j+1}$
as long as $B^{\chi}_j > 0$, which holds for all $j \leq k_0$, due
to \eqref{eq_bkhi_def} and \eqref{eq_k0_def}. This observation
implies \eqref{eq_betanew_mon}.

It remains to prove \eqref{eq_betanew_and_beta}.
We observe that, due to \eqref{eq_rho_def}, the sequence 
$0 = \rho_1, \rho_2, \dots$ grows monotonically and is bounded
from above by 1. Combined with \eqref{eq_anew_def}, this implies that
\begin{align}
A^{new}_k < \tilde{A}_k, \quad k = 1, 2, \dots.
\label{eq_anew_atilde}
\end{align}
Eventually, we show by induction that
\begin{align}
\beta^{new}_{k+1} - \beta^{new}_k \leq \beta_{k+1} - \beta_k \quad 
\text{and} \quad 
\beta^{new}_{k+1} \leq \beta_{k+1},
\label{eq_betanew_both}
\end{align}
for all $k=1,2,\dots,k_0+1$, with $k_0$ defined via \eqref{eq_k0_def}.
For $k=1,2$, the inequalities \eqref{eq_betanew_both}
hold due to \eqref{eq_betanew_def}. We assume that they hold for 
some $k \leq k_0$. First, we combine
\eqref{eq_bkhi_def},
\eqref{eq_betak_def},
\eqref{eq_k0_def},
\eqref{eq_betanew_def},
\eqref{eq_anew_atilde} and the induction hypothesis to conclude that
\begin{align}
\beta^{new}_{k+2} - \beta^{new}_{k+1} 
 =
B^{\chi}_k \cdot \beta^{new}_{k+1} + 
A^{new}_k \cdot (\beta^{new}_{k+1} - \beta^{new}_k)  \leq 
B^{\chi}_k \cdot \beta_{k+1} +
\tilde{A}_k \cdot (\beta_{k+1} - \beta_k).
\label{eq_beta_first}
\end{align}
Then, we combine
\eqref{eq_bkhi_def},
\eqref{eq_betak_def},
\eqref{eq_k0_def},
\eqref{eq_betanew_def}, 
\eqref{eq_anew_atilde}
and the induction hypothesis to conclude that
\begin{align}
\beta_{k+2}-\beta^{new}_{k+2} & \; =
 (B^{\chi}_k+1) \cdot (\beta_{k+1}-\beta^{new}_{k+1}) \; +
  \tilde{A}_k \cdot (\beta_{k+1}-\beta_k) - 
  A^{new}_k \cdot (\beta^{new}_{k+1}-\beta^{new}_k) \nonumber \\
& \; > 
\beta_{k+1}-\beta^{new}_{k+1} > 0,
\end{align}
which finishes the proof.
\end{proof}
Theorem~\ref{thm_rho} allows us to find a lower bound on $\beta_k$ by
finding a lower bound on $\beta^{new}_k$, for all $k \leq k_0+2$.
In the following theorem,
we simplify the recurrence relation \eqref{eq_betanew_def} by
rescaling $\left\{\beta^{new}_k\right\}$.
\begin{thm}
Suppose that $\chi_n > c^2+6$, and that the sequence
$\beta^{new}_1, \beta^{new}_2, \dots$ is defined
via \eqref{eq_betanew_def} in Theorem~\ref{thm_rho}.
Suppose also that the sequence $f_1,f_2,\dots$ is defined
via the formula
\begin{align}
f_k = \frac{ (4k-4) \cdot (4k-6) }{ 4k-1 },
\label{eq_fn_def}
\end{align}
for all $k=1,2,\dots$, and the sequence $\gamma_1, \gamma_2, \dots$
is defined via the formulae
\begin{align}
& \gamma_1 = \beta^{new}_1, \nonumber \\
& \gamma_k = f_k \cdot \beta^{new}_k, 
\label{eq_gamma_def}
\end{align}
for $k \geq 2$.
Then, the sequence $\gamma_1, \gamma_2, \dots$ satisfies,
for $k \geq 2$, the recurrence relation
\begin{align}
\label{eq_gamma1}
& \gamma_1 = \sqrt{2}, \\
\label{eq_gamma2}
& \gamma_2 = \frac{8 }{7 \sqrt{2}} \cdot 
             \left( 2 + 3 \cdot \frac{\chi_n-c^2}{c^2} \right), \\
\label{eq_gamma3}
& \gamma_3 = 
\frac{16\sqrt{2}}{11} \cdot \left(
  3+15\cdot\frac{\chi_n-c^2}{c^2} + 
  \frac{105}{8}\cdot\frac{\chi_n-c^2}{c^2}\cdot \frac{\chi_n-c^2-6}{c^2}-
  \frac{105}{2c^2} \right), \\
& \gamma_{k+2} = 
  \left(B^I_k+B^{II}_k\right) \cdot \gamma_{k+1} - \gamma_k,
\label{eq_gamma_rec}
\end{align}
where the sequences $\left\{ B^I_k \right\}$ and
$\left\{ B^{II}_k \right\}$ are defined via the formulae
\begin{align}
B^I_k = \frac{4 \cdot (4k+1) \cdot (4k+3)^2 }
             {4k \cdot (4k-2) \cdot (4k+7) } \cdot
\left[ \frac{\chi_n-c^2-2k\cdot(2k+1)}{c^2} \right],
\label{eq_b1_def}
\end{align}
for all $k=1,2,\dots$,
and
\begin{align}
B^{II}_k = 2 +
\frac{60}{32 k^4 + 32 k^3 - 38 k^2 + 7k},
\label{eq_b2_def}
\end{align}
for all $k=1,2,\dots$,
respectively. Moreover,
\begin{align}
\frac{245}{22} \cdot \frac{\chi_n-c^2-6}{c^2} =
B^I_1 > B^I_2 > \dots > B^I_{k_0} > 0,
\label{eq_b1_mon}
\end{align}
where $k_0$ is defined via \eqref{eq_k0_def}, and
\begin{align}
\frac{42}{11} = B^{II}_1 > B^{II}_2 > \dots > B^{II}_k > \dots > 2.
\label{eq_b2_mon}
\end{align}
\label{thm_gamma}
\end{thm}
\begin{proof}
The identity \eqref{eq_gamma1} follows immediately from 
\eqref{eq_betanew_def} and \eqref{eq_gamma_def}. Then, it follows
from \eqref{eq_bigak_def}, \eqref{eq_bigbk_def}, that
\begin{align}
A_1 = \frac{7}{6}, \quad 
B_0 = \frac{\sqrt{5}}{2} \cdot \left(\frac{3\chi_n}{c^2}-1\right)
    = \frac{\sqrt{5}}{2} \cdot \left(2 + 3 \cdot \frac{\chi_n-c^2}{c^2}\right),
\label{eq_a1b0}
\end{align}
moreover,
\begin{align}
B_1 
& \; = \frac{7\sqrt{5}}{4} \cdot \frac{\chi_n-6}{c^2} - \frac{11\sqrt{5}}{12}
     = \frac{7\sqrt{5}}{4} \cdot \frac{\chi_n-c^2-6}{c^2}
       +\frac{7\sqrt{5}}{4} - \frac{11\sqrt{5}}{12} \nonumber \\
& \;  = \frac{\sqrt{5}}{12} \cdot
       \left( 10 + 21 \cdot \frac{\chi_n-c^2-6}{c^2} \right).
\label{eq_b1}
\end{align}
We combine \eqref{eq_a1b0} with
\eqref{eq_alphak_rec}, \eqref{eq_betak_def}, \eqref{eq_betanew_def},
\eqref{eq_fn_def}, \eqref{eq_gamma_def} to conclude that
\begin{align}
\gamma_2 = \frac{8}{7} \cdot \beta_2 
= \frac{8}{7} \cdot \sqrt{\frac{2}{5}} \cdot \alpha_2
= \frac{8}{7} \cdot \sqrt{\frac{2}{5}} \cdot B_0,
\end{align}
from which \eqref{eq_gamma2} follows. Then
we combine \eqref{eq_a1b0}, \eqref{eq_b1} with
\eqref{eq_alphak_rec}, \eqref{eq_betak_def}, \eqref{eq_betanew_def},
\eqref{eq_fn_def}, \eqref{eq_gamma_def} to conclude that
\begin{align}
\gamma_3 
& \; = \frac{48}{11} \cdot \beta_3 
     = \frac{48}{11} \cdot \frac{\sqrt{2}}{3} \cdot \alpha_3
     = \frac{48\sqrt{2}}{33} \cdot(B_1 \alpha_2 - A_1 \alpha_1)
     = \frac{48\sqrt{2}}{33} \cdot(B_1 B_0 - A_1) \nonumber \\
& \; = \frac{16\sqrt{2}}{11} \cdot
       \left( \frac{5}{24} \cdot
              \left(2+3\cdot\frac{\chi_n-c^2}{c^2}\right) \cdot
              \left(10+21\cdot\frac{\chi_n-c^2-6}{c^2}\right) - \frac{7}{6}
       \right),
\end{align}
which simplifies to yield \eqref{eq_gamma3}.
The relation \eqref{eq_gamma_rec} is established by using
\eqref{eq_bkhi_def}, \eqref{eq_betanew_def}, \eqref{eq_anew_def},
\eqref{eq_fn_def}, \eqref{eq_gamma_def} to expand,
for all $k \geq 2$,
\begin{align}
\gamma_{k+2}
& \; = f_{k+2} \cdot \beta^{new}_{k+2} = 
       f_{k+2} \cdot (B^{\chi}_k+1+A^{new}_k) \cdot \beta^{new}_{k+1} -
       f_{k+2} \cdot A^{new}_k \cdot \beta^{new}_k \nonumber \\
& \; = \frac{f_{k+2}}{f_{k+1}} 
       \cdot (B^{\chi}_k+1+A^{new}_k) \cdot \gamma_{k+1} -
       \frac{f_{k+2}}{f_k} \cdot A^{new}_k \cdot \gamma_k.
\label{eq_gamma_kplus2}
\end{align}
Since, due to \eqref{eq_anew_def}, \eqref{eq_fn_def}, we have
\begin{align}
& \frac{f_{k+2}}{f_k} \cdot A^{new}_k  = \nonumber \\
&\frac{(4n+4)\cdot(4n+2)}{4n+7} \cdot \frac{4n-1}{(4n-4)\cdot(4n-6)} \cdot
\frac{(4n-4)\cdot(4n-6)\cdot(4n+7)}{(4n+4)\cdot(4n+2)\cdot(4n-1)} = 1,
\label{eq_anew_1}
\end{align}
the identity \eqref{eq_gamma_rec} readily follows from \eqref{eq_gamma_kplus2},
\eqref{eq_anew_1}, with
\begin{align}
B^I_k = \frac{f_{k+2}}{f_{k+1}} \cdot B^{\chi}_k
\label{eq_b1k_1}
\end{align}
and
\begin{align}
B^{II}_k = \frac{f_{k+2}}{f_{k+1}} \cdot \left(A^{new}_k+1\right).
\label{eq_b2k_1}
\end{align}
We substitute \eqref{eq_bkhi_def}, \eqref{eq_fn_def} into
\eqref{eq_b1k_1} to obtain \eqref{eq_b1_def}.
Next,
\begin{align}
& \frac{d}{dk} \left[ \frac{4 \cdot (4k+1) \cdot (4k+3)^2 }
             {4k \cdot (4k-2) \cdot (4k+7) } \right] =
\frac{9}{14k^2} + \frac{512}{21 \cdot(7+4k)^2} - \frac{50}{3\cdot(2k-1)^2} \;<
\nonumber \\
& \frac{1}{(k-1/2)^2} \cdot \left(
   \frac{9}{14} + \frac{512}{21 \cdot 16} - \frac{50}{12} \right) =
-\frac{2}{(k-1/2)^2} < 0,
\label{eq_b1_before}
\end{align}
for all $k \geq 1$.
Due to \eqref{eq_k0_def},
the term inside the square brackets of \eqref{eq_b1_def} is positive
for all $k \geq k_0$ and monotonically decreases as $k$ grows, which,
combined with \eqref{eq_b1_before}, implies \eqref{eq_b1_mon}.
Eventually, we substitute \eqref{eq_anew_def}, \eqref{eq_fn_def}
into \eqref{eq_b2k_1} and use \eqref{eq_anew_1} to obtain, for all $k \geq 1$,
\begin{align}
B^{II}_k = \frac{f_{k+2} + f_k}{f_{k+1}},
\end{align}
which yields \eqref{eq_b2_def} through 
straightforward algebraic manipulations. The monotonicity relation
\eqref{eq_b2_mon} follows immediately from \eqref{eq_b2_def}.
\end{proof}
We analyze the sequence $\left\{ \gamma_k \right\}$ from 
Theorem~\ref{thm_gamma} by considering the ratios
of its consecutive elements. The latter are
bounded from below by the largest eigenvalue of 
the characteristic equation of the recurrence relation \eqref{eq_gamma_rec}.
In the following two theorems, we elaborate
on these ideas.
\begin{thm}
Suppose that $\chi_n > c^2$, and that the sequence $r_1,r_2,\dots$
is defined via the formula
\begin{align}
r_k = \frac{\gamma_{k+1}}{\gamma_k},
\label{eq_r_def}
\end{align}
for all $k=1,2,\dots$, where the sequence $\gamma_1,\gamma_2,\dots$
is defined via \eqref{eq_gamma_def} in Theorem~\ref{thm_gamma}.
Suppose also that the sequence $\sigma_1,\sigma_2,\dots$
is defined via the formula
\begin{align}
\sigma_k = \frac{B^I_k+B^{II}_k}{2} +
            \sqrt{ \left( \frac{B^I_k+B^{II}_k}{2} \right)^2 - 1 },
\label{eq_sigma_def}
\end{align}
for all $k=1,2,\dots$, where $B^I_k, B^{II}_k$ are defined via
\eqref{eq_b1_def},\eqref{eq_b2_def} in Theorem~\ref{thm_gamma},
respectively. Then, 
\begin{align}
r_2 > B^I_2 + B^{II}_2.
\label{eq_r2_b2}
\end{align}
Moreover, if $B^I_2+B^{II}_2 > 2$, then $\sigma_2 > 0$, and
\begin{align}
r_2 > \sigma_2.
\label{eq_r2_sigma2}
\end{align}
\label{thm_r_sigma}
\end{thm}
\begin{proof}
We use \eqref{eq_b1_def}, \eqref{eq_b2_def} to obtain
\begin{align}
B^I_2 + B^{II}_2 = \frac{44}{21} + \frac{121}{20} \cdot 
\frac{\chi_n-c^2-20}{c^2}.
\label{eq_b2_1}
\end{align}
Next, we plug \eqref{eq_gamma2},\eqref{eq_gamma3} into \eqref{eq_r_def}
to obtain
\begin{align}
r_2 =
\frac{28}{11} \cdot
& \left(
  3+15\cdot\frac{\chi_n-c^2}{c^2} + 
  \frac{105}{8}\cdot\frac{\chi_n-c^2}{c^2}\cdot \frac{\chi_n-c^2-6}{c^2}-
  \frac{105}{2c^2} \right) \cdot \nonumber \\
& \left( 2 + 3 \cdot \frac{\chi_n-c^2}{c^2} \right)^{-1}.
\label{eq_r2_1}
\end{align}
We subtract \eqref{eq_b2_1} from \eqref{eq_r2_1} to obtain, by performing
elementary algebraic manipulations,
\begin{align}
r_2 - (B^I_2 + B^{II}_2) & \; = 
\frac{247}{77} + \frac{1119}{220} \cdot \frac{\chi_n-c^2}{c^2} -
\frac{98}{33} \cdot \left( 2 + 3 \cdot \frac{\chi_n-c^2}{c^2} \right)^{-1}
+ \frac{596}{11c^2} \nonumber \\
& \; > \frac{247}{77} - \frac{98}{66} = \frac{398}{231} > 0,
\end{align}
which implies \eqref{eq_r2_b2}. Due to \eqref{eq_sigma_def},
$\sigma_2$ is positive if and only if $B^I_2+B^{II}_2 > 2$; in that case,
\begin{align}
B^I_{2}+B^{II}_2 > \sigma_2,
\end{align}
which, combined with \eqref{eq_r2_b2}, implies \eqref{eq_r2_sigma2}.
\end{proof}
The following theorem extends Theorem~\ref{thm_r_sigma}.
\begin{thm}
Suppose that $\chi_n > c^2$, and that $k_0>2$, where $k_0$ is defined via
\eqref{eq_k0_def} in Theorem~\ref{thm_monotone}. Suppose
also that the sequences $r_1,r_2, \dots$ and $\sigma_1,\sigma_2,\dots$
are defined, respectively,
via \eqref{eq_r_def}, \eqref{eq_sigma_def}
in Theorem~\ref{thm_r_sigma}. Then,
\begin{align}
\sigma_1 > \sigma_2 > \sigma_3 > \dots > \sigma_{k_0} > 1.
\label{eq_sigma_mon}
\end{align}
In addition,
\begin{align}
r_2 > r_3 > \dots > r_{k_0} > 1.
\label{eq_r_mon}
\end{align}
Moreover,
\begin{align}
r_2 > \sigma_2 > 1, \quad r_3 > \sigma_3 > 1, \quad \dots, \quad 
r_{k_0} > \sigma_{k_0} > 1.
\label{eq_r_sigma_ineq}
\end{align}
\label{thm_r_sigma_ineq}
\end{thm}
\begin{proof}
We combine \eqref{eq_b1_def}, \eqref{eq_b2_def}, \eqref{eq_b1_mon},
\eqref{eq_b2_mon} in Theorem~\ref{thm_gamma} 
with \eqref{eq_sigma_def} in Theorem~\ref{thm_r_sigma}
to conclude that,
for all $k=1,2,\dots,k_0$,
\begin{align}
\sigma_k > \frac{B^I_k+B^{II}_k}{2} > \frac{B^{II}_k}{2} > 1.
\label{eq_sigma_pos}
\end{align}
We use this in combination with \eqref{eq_b1_mon} and \eqref{eq_b2_mon}
to conclude that \eqref{eq_sigma_mon} holds.
Then, we use \eqref{eq_sigma_pos}
and Theorem~\ref{thm_r_sigma} to conclude that
\begin{align}
r_2 > \sigma_2 > 1.
\label{eq_r2_sigma2_1}
\end{align}
Next, we prove
\eqref{eq_r_sigma_ineq} by induction on $k \leq k_0$.
The case $k=2$ is handled by \eqref{eq_r2_sigma2_1}. Suppose that
$2 < k < k_0$, and \eqref{eq_r_sigma_ineq} is true
for $k$, i.e.
\begin{align}
r_k > \sigma_k > 1.
\label{eq_rk_sigmak}
\end{align}
 We consider the quadratic equation
\begin{align}
x^2 - (B^I_k + B^{II}_k) \cdot x + 1 = 0,
\label{eq_quadr}
\end{align}
in the unknown $x$.
Due to \eqref{eq_sigma_def} and \eqref{eq_sigma_pos}, 
$\sigma_k$ is the largest root of the quadratic equation
\eqref{eq_quadr}, and, moreover, $\sigma_k^{-1} < 1$ is its
second (smallest) root. Thus, the left hand side of \eqref{eq_quadr}
is negative if and only if $x \in (\sigma_k^{-1}, \sigma_k)$.
We combine this observation with
\eqref{eq_rk_sigmak} to conclude that
\begin{align}
r_k^2 - (B^I_k + B^{II}_k) \cdot r_k + 1 > 0,
\end{align}
and, consequently,
\begin{align}
r_k > (B^I_k + B^{II}_k) - \frac{1}{r_k}.
\label{eq_r_ric}
\end{align}
Then, we substitute \eqref{eq_r_def} into \eqref{eq_gamma_rec} to obtain
\begin{align}
r_{k+1} = \frac{\gamma_{k+2}}{\gamma_{k+1}}
        = \frac{(B^I_k + B^{II}_k) \cdot \gamma_{k+1} - \gamma_k}{\gamma_{k+1}}
        = (B^I_k + B^{II}_k) - \frac{1}{r_k}.
\label{eq_r_rec}
\end{align}
By combining \eqref{eq_r_ric} with \eqref{eq_r_rec} we conclude that
\begin{align}
r_k > r_{k+1}.
\label{eq_rk_rkp}
\end{align}
Moreover, we combine \eqref{eq_rk_sigmak} with \eqref{eq_r_rec} 
and use the fact that $\sigma_k$ is a root of \eqref{eq_quadr} to obtain
the inequality
\begin{align}
r_{k+1} = (B^I_k + B^{II}_k) - \frac{1}{r_k} >
          (B^I_k + B^{II}_k) - \frac{1}{\sigma_k} = \sigma_k.
\label{eq_rkp_sigmak}
\end{align}
However, combined with the already proved \eqref{eq_sigma_mon} and
the fact that $k < k_0$, 
the inequality \eqref{eq_rkp_sigmak} implies that
\begin{align}
r_{k+1} > \sigma_{k+1}.
\label{eq_rkp_sigmakp}
\end{align}
This completes the proof of \eqref{eq_r_sigma_ineq}. The relation
\eqref{eq_r_mon} follows from the inequality \eqref{eq_rk_rkp}
above.
\end{proof}
In the following theorem, we 
bound the product of several $\sigma_k$'s
by a definite integral.
\begin{thm}
Suppose that $\chi_n > c^2$, and that $k_0 > 2$, where
$k_0$ is defined via \eqref{eq_k0_def} in Theorem~\ref{thm_monotone}.
Suppose also that the real valued function $g_n$ is
defined via the formula
\begin{align}
g_n(x) =
 1 + 
         2 \cdot \left( \frac{\chi_n-c^2}{c^2} - 
                        \left(\frac{2x}{c}\right)^2 \right) +
  \sqrt{ \left[
         1 + 
         2 \cdot \left( \frac{\chi_n-c^2}{c^2} - 
                        \left(\frac{2x}{c}\right)^2 \right)    
         \right]^2 - 1
  },
\label{eq_gn_def}
\end{align}
for the real values of $x$ satisfying the inequality $4x^2 \leq \chi_n - c^2$.
Suppose furthermore that the sequence $\sigma_1, \sigma_2, \dots$
is defined via the formula \eqref{eq_sigma_def} in Theorem~\ref{thm_r_sigma}.
Then,
\begin{align}
\sigma_2 \cdot \sigma_3 \cdot \dots \cdot \sigma_{k_0-1} > 
(g_n(0))^{-4} \cdot 
\exp \int_0^{\left(\sqrt{\chi_n - c^2}\right)/2}
     \log \left( g_n(x) \right) \; dx.
\label{eq_sigma_prod}
\end{align}
\label{thm_product}
\end{thm}
\begin{proof}
We observe that, for all $k=1,2,\dots$,
\begin{align}
4\cdot k^2 < 2k \cdot (2k+1) < 4 \cdot (k+1)^2 < 2(k+1) \cdot (2(k+1)+1).
\label{eq_k_ineq}
\end{align}
In combination with \eqref{eq_k0_def}, this implies that,
for all $k=1,\dots,k_0$,
\begin{align}
\chi_n - c^2 - 4 \cdot k^2 > 0.
\label{eq_khi_ineq4}
\end{align}
Moreover, due to \eqref{eq_b1_def}, \eqref{eq_b2_def} 
in Theorem~\ref{thm_gamma}, the inequality
\begin{align}
2 < 2 + 4 \cdot 
    \left( \frac{\chi_n-c^2}{c^2} - 
           \left( \frac{2\cdot(k+1)}{c} \right)^2 \right) 
  < B^I_k + B^{II}_k
\label{eq_bster_ineq}
\end{align}
holds for all $k=1,\dots,k_0-1$, where $B^I_k, B^{II}_k$ are defined
via \eqref{eq_b1_def}, \eqref{eq_b2_def}, respectively.
We combine \eqref{eq_bster_ineq} with \eqref{eq_sigma_def}
in Theorem~\ref{thm_r_sigma} and \eqref{eq_gn_def} above to obtain
the inequality
\begin{align}
\sigma_k > g_n(k+1),
\label{eq_sigma_gn}
\end{align}
which holds for all $k=1,\dots,k_0-1$. Consequently, using the monotonicity
of $g_n$,
\begin{align}
& \sigma_2 \cdot \sigma_3 \cdot \dots \cdot \sigma_{k_0-1} > \nonumber \\
& g_n(3) \cdot g_n(4) \cdot \dots \cdot g_n(k_0) = 
 \frac{g_n(0) \cdot g_n(1) \cdot \dots \cdot g_n(k_0-1) \cdot g_n(k_0)^2}
       {g_n(0) \cdot g_n(1) \cdot g_n(2) \cdot g_n(k_0) } > \nonumber \\
& g_n(0)^{-4} \cdot 
 \exp\left( \log(g_n(0)) + \dots + \log(g_n(k_0+1)) +
2 \cdot \log(g_n(k_0)) \right).
\label{eq_gn_1}
\end{align}
Obviously, due to \eqref{eq_khi_ineq4}, the inequality
\begin{align}
\log(g_n(k)) > \int_k^{k+1} \log(g_n(x)) \; dx
\label{eq_gn_2}
\end{align}
holds for all $k=0,\dots,k_0-1$. Next, due to \eqref{eq_k0_def} and 
\eqref{eq_k_ineq}, we have
\begin{align}
k_0 < \frac{1}{2} \sqrt{\chi_n-c^2} < k_0 + 2.
\label{eq_k0_ineq}
\end{align}
Therefore,
\begin{align}
2 \cdot \log(g_n(k_0)) 
> \left( \frac{1}{2} \sqrt{\chi_n-c^2} - k_0 \right)
\cdot \log(g_n(k_0)) > 
\int_{k_0}^{\left(\sqrt{\chi_n - c^2}\right)/2} g_n(x) \; dx.
\label{eq_gn_3}
\end{align}
Thus, the inequality \eqref{eq_sigma_prod} follows from
the combination of \eqref{eq_gn_1}, \eqref{eq_gn_2} and \eqref{eq_gn_3}.
\end{proof}
\subsection{Principal Result}
\label{sec_principal}
In this subsection, we use the tools developed in
Section~\ref{sec_expansion} to derive an upper bound
on $\left|\lambda_n\right|$. Theorem~\ref{thm_big_inequality}
is the principal result of this subsection.

In the following theorem, we simplify the integral in \eqref{eq_sigma_prod}
by expressing it in terms of elliptic functions.
\begin{thm}
Suppose that $\chi_n > c^2$, and that the real-valued function $g_n$
is defined via the formula \eqref{eq_gn_def} in Theorem~\ref{thm_product}.
Then,
\begin{align}
\int_0^{\left(\sqrt{\chi_n - c^2}\right)/2}
     \log \left( g_n(x) \right) \; dx =
\frac{\chi_n-c^2}{c} \cdot
\int_0^{\pi/2} 
\frac{ \sin^2 (\theta) \; d\theta }
     { \sqrt{1 + \frac{\chi_n-c^2}{c^2} \cdot \cos^2(\theta) } }.
\label{eq_integral}
\end{align}
Moreover,
\begin{align}
\int_0^{\left(\sqrt{\chi_n - c^2}\right)/2}
     \log \left( g_n(x) \right) \; dx =
\sqrt{\chi_n} \cdot \left[
F\left( \sqrt{\frac{\chi_n-c^2}{\chi_n}} \right) -
E\left( \sqrt{\frac{\chi_n-c^2}{\chi_n}} \right)
\right],
\label{eq_integral_fme}
\end{align}
where $F,E$ are the elliptic integrals defined, respectively,
via the formula \eqref{eq_F}, \eqref{eq_E} in Section~\ref{sec_elliptic}.
\label{thm_integral}
\end{thm}
\begin{proof}
We use \eqref{eq_gn_def} and perform the change of variable
\begin{align}
s = \frac{2x}{\sqrt{\chi_n-c^2}}
\label{eq_s_def}
\end{align}
in the left-hand side of \eqref{eq_integral} to obtain
\begin{align}
\int_0^{\left(\sqrt{\chi_n - c^2}\right)/2}
     \log \left( g_n(x) \right) \; dx & \; =
 \frac{\sqrt{\chi_n-c^2}}{2} \cdot
\int_0^1 \log\left( g_n\left(\frac{s \sqrt{\chi_n-c^2}}{2}  \right)\right)\;ds
\nonumber \\
& \; = \frac{V\cdot c}{2} \cdot
  \int_0^1 \log\left(
     1 + 2V^2(1-s^2) + \sqrt{(1 + 2V^2(1-s^2))^2-1}
  \right) \; ds \nonumber \\
& \; = \frac{V\cdot c}{2} \cdot \int_0^1 \log(h(s)) \; ds,
\label{eq_int_1}
\end{align}
where $V$ is defined via the formula
\begin{align}
V = \sqrt{ \frac{\chi_n-c^2}{c^2} },
\label{eq_v_def}
\end{align}
and the function $h:\left[0,1\right] \to \Rc$ is defined via the formula
\begin{align}
h(s) = 1 + 2V^2(1-s^2) + \sqrt{(1 + 2V^2(1-s^2))^2-1}.
\label{eq_h_def}
\end{align}
We observe that $\log(h(1)) = 0$ and $h(0)$ is finite, hence
\begin{align}
\int_0^1 \log(h(s)) \; ds = \left[ s \cdot \log(h(s)) \right]_0^1
-\int_0^1 \frac{s \cdot h'(s)}{h(s)} \; ds
=
-\int_0^1 \frac{s \cdot h'(s)}{h(s)} \; ds.
\label{eq_int_2}
\end{align}
Then, we differentiate $h(s)$, defined via \eqref{eq_h_def},
with respect to $s$ to obtain
\begin{align}
h'(s) 
& \; = -2V^2 \cdot 2s + 
\frac{ 2\cdot(1+2V^2(1-s^2)) \cdot (-2V^2 \cdot 2s) }
     { 2 \sqrt{ (1+2V^2(1-s^2))^2-1 } } \nonumber \\
& \; = -4V^2s \cdot
\left(1 + 
  \frac{ 1+2V^2(1-s^2) }
     { \sqrt{ (1+2V^2(1-s^2))^2-1 } }
\right)
= - \frac{ 4V^2s \cdot h(s) }{ \sqrt{ (1+2V^2(1-s^2))^2-1 } }.
\label{eq_dh}
\end{align}
We substitute \eqref{eq_dh} into \eqref{eq_int_2} to obtain 
\begin{align}
\int_0^1 \log(h(s)) \; ds & \; = 
 \int_0^1 \frac{ 4V^2s^2 }{ \sqrt{ (1+2V^2(1-s^2))^2-1 } } \; ds
  \nonumber \\
& \; =\int_0^1 \frac{ 4V^2s^2 }{ \sqrt{ 4V^4(1-s^2)^2+4V^2(1-s^2) } } \; ds
  \nonumber \\
& \; =
2V \cdot \int_0^1 \frac{s^2}{\sqrt{ (1-s^2) \cdot(1 + V^2(1-s^2)) }} \; ds.
\label{eq_int_3}
\end{align}
We perform the change of variable
\begin{align}
s = \sin(\theta), \quad ds = \cos(\theta) \cdot d\theta,
\label{eq_theta}
\end{align}
to transform \eqref{eq_int_3} into
\begin{align}
\int_0^1 \log(h(s)) \; ds =
2V \cdot \int_0^{\pi/2} 
\frac{ \sin^2 (\theta) \; d\theta }{ \sqrt{1 + V^2 \cdot \cos^2(\theta)} }.
\label{eq_int_4}
\end{align}
We combine \eqref{eq_int_1}, \eqref{eq_v_def} and \eqref{eq_int_4}
to obtain the formula \eqref{eq_integral}. Next,
we express \eqref{eq_integral} in terms of the elliptic integrals $F(k)$
and $E(k)$,
defined, respectively, 
via \eqref{eq_F},\eqref{eq_E} in Section~\ref{sec_elliptic}. We note that
\begin{align}
F(k)-E(k)  = 
\int_0^{\pi/2} \frac{ k^2 \sin^2 t \; dt }{\sqrt{1-k^2 \sin^2 t}} =
\frac{k^2}{\sqrt{1-k^2}} \cdot 
\int_0^{\pi/2} \frac{\sin^2 t \; dt}
 {\sqrt{1 + \frac{k^2}{1-k^2} \cdot \cos^2 t }}.
\label{eq_fme}
\end{align}
Motivated by \eqref{eq_integral} and \eqref{eq_fme},
we solve the equation
\begin{align}
\frac{k^2}{1-k^2} = \frac{\chi_n - c^2}{c^2}
\end{align}
in the unknown $k$,
to obtain the solution
\begin{align}
k = \sqrt{\frac{\chi_n - c^2}{\chi_n}}.
\label{eq_ksquare}
\end{align}
We plug \eqref{eq_ksquare} into \eqref{eq_fme} to conclude that
\begin{align}
& F\left( \sqrt{\frac{\chi_n-c^2}{\chi_n}} \right) -
E\left( \sqrt{\frac{\chi_n-c^2}{\chi_n}} \right) =
\frac{\chi_n-c^2}{c\sqrt{\chi_n}} \cdot
\int_0^{\pi/2} 
\frac{ \sin^2 (\theta) \; d\theta }
     { \sqrt{1 + \frac{\chi_n-c^2}{c^2} \cdot \cos^2(\theta) } }.
\label{eq_int_5}
\end{align}
We combine \eqref{eq_integral} with \eqref{eq_int_5} to 
obtain \eqref{eq_integral_fme}.
\end{proof}
In the following theorem, 
we establish a relationship between the eigenvalue
$\lambda_n$ of the integral operator $F_c$ defined via
\eqref{eq_pswf_fc} in Section~\ref{sec_pswf}, and the value of
$a^{(n,c)}_1$ defined via \eqref{eq_aknc_def} in Theorem~\ref{thm_aknc}.
\begin{thm}
Suppose that $n>0$ is an even integer number, and that
$\lambda_n$ is the $n$th eigenvalue of the integral
operator $F_c$ defined via \eqref{eq_pswf_fc} in Section~\ref{sec_pswf}.
In other words, $\lambda_n$ satisfies the identity
\eqref{eq_prolate_integral} in Section~\ref{sec_pswf}. Suppose also,
that the sequence $a^{(n,c)}_1, a^{(n,c)}_2, \dots$ is defined
via the formula \eqref{eq_aknc_def} in Theorem~\ref{thm_aknc}. Then,
\begin{align}
\lambda_n = \frac{\sqrt{2}}{\psi_n(0)} \cdot a^{(n,c)}_1,
\label{eq_lambda_a}
\end{align}
where $\psi_n$ is the $n$th prolate spheroidal wave function 
defined in Section~\ref{sec_pswf}.
\label{thm_lambda_a}
\end{thm}
\begin{proof}
Due to \eqref{eq_prolate_integral} in Section~\ref{sec_pswf},
\eqref{eq_legendre_pol_0_1}, \eqref{eq_legendre_normalized}
in Section~\ref{sec_legendre}, and \eqref{eq_aknc_def} above,
\begin{align}
\lambda_n \cdot \psi_n(0) = \int_{-1}^1 \psi_n(t) \; dt
= \sqrt{2} \cdot \int_{-1}^1 \psi_n(t) \cdot \overline{P_0}(t) \; dt
= \sqrt{2} \cdot a^{(n,c)}_1,
\end{align}
from which \eqref{eq_lambda_a} readily follows.
\end{proof}
In the following theorem,
we provide an upper bound on $| \lambda_n |$
in terms of the elements of the sequence $\left\{ \gamma_k \right\}$,
defined via \eqref{eq_gamma_def} in Theorem~\ref{thm_gamma} above.
\begin{thm}
Suppose that $n>0$ is an even integer number, 
and that $\lambda_n$ is the $n$th eigenvalue of the
integral operator $F_c$,
defined via \eqref{eq_pswf_fc}, \eqref{eq_prolate_integral}
in Section~\ref{sec_pswf}. 
Suppose also that $\chi_n>c^2$, and that 
$k_0>2$, where $k_0$ is defined via \eqref{eq_k0_def} in
Theorem~\ref{thm_monotone}. Suppose furthermore, that
the sequence $\gamma_1, \gamma_2, \dots$ is defined
via \eqref{eq_gamma_def} in Theorem~\ref{thm_gamma}. Then,
\begin{align}
| \lambda_n | < 
\frac{2}{\left|\psi_n(0)\right|} \cdot
\frac{(4k_0-4)\cdot(4k_0-6)}{(4k_0-1)\cdot\sqrt{4k_0-3}} \cdot
\frac{1}{\gamma_{k_0}}.
\label{eq_lambda_gamma}
\end{align}
\label{thm_lambda_gamma}
\end{thm}
\begin{proof} 
We combine the inequality \eqref{eq_a1_alphak_ineq} in 
Theorem~\ref{thm_alphak} with the identity \eqref{eq_lambda_a}
in Theorem~\ref{thm_lambda_a}, to conclude that
\begin{align}
| \lambda_n | = \frac{\sqrt{2}}{| \psi_n(0) |} \cdot |a^{(n,c)}_1| <
\frac{\sqrt{2}}{| \psi_n(0) |} \cdot \frac{1}{\alpha_{k_0}} =
\frac{2}{| \psi_n(0) |} \cdot \frac{1}{\sqrt{4k_0-3}} \cdot 
\frac{1}{\beta_{k_0}},
\label{eq_laga_1}
\end{align}
where $\beta_{k_0}$ is defined via \eqref{eq_betak_def}
in Theorem~\ref{thm_betak}. Next, we combine
\eqref{eq_betanew_def}, \eqref{eq_betanew_and_beta}
in Theorem~\ref{thm_rho}, 
\eqref{eq_fn_def},\eqref{eq_gamma_def} in Theorem~\ref{thm_gamma},
and \eqref{eq_laga_1} to obtain
the inequality
\begin{align}
| \lambda_n | & \; < 
\frac{2}{| \psi_n(0) |} \cdot \frac{1}{\sqrt{4k_0-3}} \cdot 
\frac{1}{\beta_{k_0}} \leq
\frac{2}{| \psi_n(0) |} \cdot \frac{1}{\sqrt{4k_0-3}} \cdot 
\frac{1}{\beta^{new}_{k_0}} \nonumber \\
& \; =
\frac{2}{| \psi_n(0) |} \cdot
\frac{(4k_0-4)\cdot(4k_0-6)}{(4k_0-1)\cdot\sqrt{4k_0-3}} \cdot
\frac{1}{\gamma_{k_0}},
\label{eq_laga_2}
\end{align}
which is precisely \eqref{eq_lambda_gamma}.
\end{proof}
The following theorem is a direct consequence of 
Theorems~\ref{thm_n_and_khi}, 
\ref{thm_n_khi_simple} in Section~\ref{sec_pswf}.
\begin{thm}
Suppose that $n > 0$ is a positive integer. Suppose also that
$n > (2c/\pi)+\sqrt{42}$.
Then,
\begin{align}
\chi_n > c^2 + 42,
\label{eq_khi_2}
\end{align}
and also,
\begin{align}
k_0 > 2,
\label{eq_k0_gt_2}
\end{align}
where $k_0$ is defined via \eqref{eq_k0_def} in
Theorem~\ref{thm_monotone}.
\label{thm_khi_2}
\end{thm}
\begin{proof}
Suppose that $c^2 < \chi_n \geq c^2 + 2$. Then, due to 
Theorem~\ref{thm_n_and_khi},
\begin{align}
n & \; < \frac{2}{\pi} \int_0^1 \sqrt{ \frac{\chi_n-c^2t^2}{1-t^2} } \; dt
  \leq \frac{2}{\pi} \int_0^1 \sqrt{c^2+\frac{42}{1-t^2} } \; dt \nonumber\\
  & \; < \frac{2c}{\pi} + \frac{2\sqrt{42}}{\pi} \cdot
                        \int_0^1 \frac{dt}{\sqrt{1-t^2}}
       = \frac{2c}{\pi} + \sqrt{42}.
\label{eq_khi_2a}
\end{align}
We combine \eqref{eq_khi_2a} with Theorem~\ref{thm_n_and_khi}
to conclude \eqref{eq_khi_2}. Then, we combine \eqref{eq_khi_2}
with \eqref{eq_k0_def} in Theorem~\ref{thm_monotone}
to conclude \eqref{eq_k0_gt_2}.
\end{proof}

The following theorem is the principal result of this paper.
\begin{thm}
Suppose that $n>0$ is an even integer number, 
and that $\lambda_n$ is the $n$th eigenvalue of the
integral operator $F_c$,
defined via \eqref{eq_pswf_fc}, \eqref{eq_prolate_integral}
in Section~\ref{sec_pswf}. 
Suppose also that $\chi_n>c^2+42$.
Suppose furthermore that the real number $\zeta(n,c)$ is defined
via the formula
\begin{align}
\zeta(n,c) = & \;
\frac{7}{2 |\psi_n(0)|} \cdot
\frac{ \left(4 \cdot \chi_n/c^2- 2 \right)^{4} }
     {       3 \cdot \chi_n/c^2 - 1            }  \cdot
\left(\chi_n-c^2\right)^{\frac{1}{4}} \cdot \nonumber \\
& \; \exp\left[
-\sqrt{\chi_n} \cdot \left(
F\left( \sqrt{\frac{\chi_n-c^2}{\chi_n}} \right) -
E\left( \sqrt{\frac{\chi_n-c^2}{\chi_n}} \right)
\right)
\right],
\label{eq_zeta_n_c}
\end{align}
where $F,E$ are the complete elliptic integrals,
defined, respectively, via \eqref{eq_F}, \eqref{eq_E}
in Section~\ref{sec_elliptic}.
Then,
\begin{align}
| \lambda_n | < \zeta(n,c).
\label{eq_big_inequality}
\end{align}
\label{thm_big_inequality}
\end{thm}
\begin{proof}
We start with observing that, due to \eqref{eq_k0_def}
in Theorem~\ref{thm_monotone} and
\eqref{eq_k0_ineq} in Theorem~\ref{thm_product},
the inequality $\chi_n>c^2+42$ implies that $k_0>2$.
We combine \eqref{eq_gamma_def} in Theorem~\ref{thm_gamma},
\eqref{eq_r_def}, \eqref{eq_sigma_def} in Theorem~\ref{thm_r_sigma}
and \eqref{eq_r_sigma_ineq} in Theorem~\ref{thm_r_sigma_ineq},
to obtain the inequality
\begin{align}
\gamma_{k_0} & \; =
\gamma_2 \cdot \frac{\gamma_3}{\gamma_2} \cdot \dots \cdot
               \frac{\gamma_{k_0-1}}{\gamma_{k_0-2}} \cdot
               \frac{\gamma_{k_0}}{\gamma_{k_0-1}} 
 = \gamma_2 \cdot r_2 \cdot \dots \cdot r_{k_0-2} \cdot r_{k_0-1}
> \gamma_2 \cdot \left( \sigma_2 \cdot \dots \cdot \sigma_{k_0-1} \right).
\label{eq_laint_1}
\end{align}
Next, we substitute \eqref{eq_gn_def}, \eqref{eq_sigma_prod} in
Theorem~\ref{thm_product} into \eqref{eq_laint_1}
to obtain the inequality
\begin{align}
\gamma_{k_0} & \; >
\gamma_2 \cdot \left(g_n(0)\right)^{-4} \cdot
\exp \int_0^{\left(\sqrt{\chi_n - c^2}\right)/2}
     \log \left( g_n(x) \right) \; dx 
\nonumber \\
& \; > \gamma_2 \cdot \left(2 + 4 \cdot \frac{\chi_n-c^2}{c^2}\right)^{-4}
 \cdot
\exp \int_0^{\left(\sqrt{\chi_n - c^2}\right)/2}
     \log \left( g_n(x) \right) \; dx,
\label{eq_laint_2}
\end{align}
where the function $g_n$ is defined via \eqref{eq_gn_def}. Then,
we plug the identity \eqref{eq_integral} from Theorem~\ref{thm_integral}
into \eqref{eq_laint_2} to obtain the inequality
\begin{align}
\frac{1}{\gamma_{k_0}} < & \;  
\frac{1}{\gamma_2} \cdot
\left(2 + 4 \cdot \frac{\chi_n-c^2}{c^2}\right)^{4} \cdot  \exp\left[
-\frac{\chi_n-c^2}{c} \cdot
\int_0^{\pi/2} 
\frac{ \sin^2 (\theta) \; d\theta }
     { \sqrt{1 + \frac{\chi_n-c^2}{c^2} \cdot \cos^2(\theta) } }\right].
\label{eq_laint_3}
\end{align}
We use \eqref{eq_k0_def} in Theorem~\ref{thm_monotone}
and \eqref{eq_k0_ineq} in Theorem~\ref{thm_product}
to conclude that
\begin{align}
\frac{(4k_0-4)\cdot(4k_0-6)}{(4k_0-1)\cdot\sqrt{4k_0-3}} <
\sqrt{4k_0} < \sqrt{2} \cdot \left(\chi_n-c^2\right)^{\frac{1}{4}}.
\label{eq_laint_4}
\end{align}
We substitute \eqref{eq_laint_4} into \eqref{eq_lambda_gamma}
in Theorem~\ref{thm_lambda_gamma} to obtain
\begin{align}
| \lambda_n | < 
\frac{2}{|\psi_n(0)|} \cdot
\sqrt{2} \cdot \left(\chi_n-c^2\right)^{\frac{1}{4}} \cdot
\frac{1}{\gamma_{k_0}}.
\label{eq_laint_5}
\end{align}
Finally, we combine \eqref{eq_gamma2} in Theorem~\ref{thm_gamma} with
\eqref{eq_laint_3}, \eqref{eq_laint_5} to obtain
\begin{align}
| \lambda_n | < & \;
\frac{7}{2|\psi_n(0)|} \cdot
\left(\chi_n-c^2\right)^{\frac{1}{4}} \cdot
\left(2 + 3 \cdot \frac{\chi_n-c^2}{c^2} \right)^{-1} \cdot
\left(2 + 4 \cdot \frac{\chi_n-c^2}{c^2}\right)^{4} \cdot \nonumber \\
& \; \exp\left[ -
\frac{\chi_n-c^2}{c} \cdot
\int_0^{\pi/2} 
\frac{ \sin^2 (\theta) \; d\theta }
     { \sqrt{1 + \frac{\chi_n-c^2}{c^2} \cdot \cos^2(\theta) } } \right].
\label{eq_laint_6}
\end{align}
Eventually, we combine \eqref{eq_integral_fme} in Theorem~\ref{thm_integral}
with \eqref{eq_laint_6} to conclude \eqref{eq_big_inequality}.
\end{proof}
\begin{remark}
The assumptions of 
Theorem~\ref{thm_big_inequality} are satisfied if $n$ is
an even integer such that
\begin{align}
n > \frac{2c}{\pi} + \sqrt{42},
\end{align}
since, in this case, $\chi_n > c^2 + 42$ due to Theorem~\ref{thm_khi_2}.
\label{rmk_42}
\end{remark}
\subsection{Weaker But Simpler Bounds}
\label{sec_weaker}
In this subsection, we use Theorem~\ref{thm_big_inequality}
in Section~\ref{sec_principal} to derive several
upper bounds on $|\lambda_n|$. While these bounds
are weaker than $\zeta(n,c)$ defined via \eqref{eq_zeta_n_c},
they have a simpler form, and contribute to a better
understanding of the decay of $|\lambda_n|$.
The principal results of this subsection are
Theorems~\ref{thm_simple_inequality}, \ref{thm_crude_inequality}.

In the following theorem, we simplify the inequality 
\eqref{eq_big_inequality}.
The resulting upper bound on $|\lambda_n|$
is weaker than \eqref{eq_big_inequality}
in Theorem~\ref{thm_big_inequality}, but has a simpler form.
\begin{thm}
Suppose that $n>0$ is an even integer number, 
and that $\lambda_n$ is the $n$th eigenvalue of the
integral operator $F_c$,
defined via \eqref{eq_pswf_fc}, \eqref{eq_prolate_integral}
in Section~\ref{sec_pswf}. 
Suppose also that $\chi_n>c^2+42$. Suppose furthermore that
the real number $\eta(n,c)$ is defined via the formula
\begin{align}
\eta(n,c) = & \;
18 \cdot (n+1) \cdot \left( \frac{\pi \cdot (n+1)}{c} \right)^7 \cdot
\nonumber \\
& \; \exp\left[
-\sqrt{\chi_n} \cdot \left(
F\left( \sqrt{\frac{\chi_n-c^2}{\chi_n}} \right) -
E\left( \sqrt{\frac{\chi_n-c^2}{\chi_n}} \right)
\right)
\right],
\label{eq_eta_n_c}
\end{align}
where $F,E$ are the complete elliptic integrals,
defined, respectively, via \eqref{eq_F}, \eqref{eq_E}
in Section~\ref{sec_elliptic}.
Then,
\begin{align}
| \lambda_n | < \eta(n,c).
\label{eq_simple_inequality}
\end{align}
\label{thm_simple_inequality}
\end{thm}
\begin{proof}
We use \eqref{eq_khi_n_square_prop} in Theorem~\ref{thm_khi_n_square}
in Section~\ref{sec_pswf} to conclude that
\begin{align}
\left( \chi_n-c^2 \right)^{1/4} < (\chi_n)^{1/4} < 
\left( \frac{\pi}{2} \cdot (n+1) \right)^{1/2}.
\label{eq_simple_1}
\end{align}
Next,
\begin{align}
\left(2 + 3 \cdot \frac{\chi_n-c^2}{c^2} \right)^{-1} \cdot
\left(2 + 4 \cdot \frac{\chi_n-c^2}{c^2}\right)^{4} <
2^7 \cdot \left( \frac{\chi_n}{c^2} \right)^3 .
\label{eq_simple_2}
\end{align}
We combine Theorems~\ref{thm_khi_n_square}, \ref{thm_psi0}
in Section~\ref{sec_pswf}
with
\eqref{eq_simple_1}, \eqref{eq_simple_2} to conclude that
\begin{align}
& \frac{1}{|\psi_n(0)|} \cdot
\frac{ \left(4 \cdot \chi_n/c^2- 2 \right)^{4} }
     {       3 \cdot \chi_n/c^2 - 1            }  \cdot
\left(\chi_n-c^2\right)^{\frac{1}{4}} < \nonumber \\
& 4 \cdot \sqrt{n \cdot \frac{\chi_n}{c^2}} \cdot
\frac{ \left(4 \cdot \chi_n/c^2- 2 \right)^{4} }
     {       3 \cdot \chi_n/c^2 - 1            }  \cdot
\left(\chi_n-c^2\right)^{\frac{1}{4}} < \nonumber \\
& 4 \cdot (n+1)^{1/2} \cdot 2^7 \cdot \left( \frac{\chi_n}{c^2} \right)^{7/2}
 \cdot \left( \frac{\pi}{2} \cdot (n+1) \right)^{1/2} < \nonumber \\
& 4 \cdot \sqrt{\frac{\pi}{2}} \cdot 2^7 \cdot (n+1) \cdot 
\left( \frac{ \pi \cdot (n+1) }{ 2c } \right)^7 =
  \sqrt{\frac{\pi}{2}} \cdot (n+1) \cdot 
  \left( \frac{ \pi \cdot (n+1) }{ c } \right)^7.
\label{eq_simple_3}
\end{align}
We conclude by combining the inequality \eqref{eq_big_inequality}
in Theorem~\ref{thm_big_inequality} above with 
the inequality \eqref{eq_simple_3}.
\end{proof}
Both $\zeta(n,c)$ and $\eta(n,c)$, defined, respectively,
via \eqref{eq_zeta_n_c} in Theorem~\ref{thm_big_inequality}
and \eqref{eq_eta_n_c} in Theorem~\ref{thm_simple_inequality},
contain an exponential term (of the form $\exp\left[ \dots \right]$).
This term depends on band limit $c$ and prolate index $n$
through $\chi_n$, which somewhat obscures its behavior.
The following theorem eliminates this inconvenience.
\begin{thm}
Suppose that $n$ is a positive integer such that
$n > 2c/\pi$, and that the function
$f:[0,\infty) \to \Rc$ is defined via the formula
\begin{align}
f(x) = -1 + \int_0^{\pi/2} \sqrt{ x + \cos^2(\theta)} \; d\theta.
\label{eq_f_def}
\end{align}
Suppose also that the function $H: [0,\infty) \to \Rc$ is the inverse of $f$,
in other words, 
\begin{align}
y = f(H(y)) = 
-1 + \int_0^{\pi/2} \sqrt{ H(y) + \cos^2(\theta)} \; d\theta,
\label{eq_big_h_def}
\end{align}
for all $y \geq 0$. Suppose furthermore that the function 
$G:[0,\infty)\to\Rc$ is defined via the formula
\begin{align}
G(x) = 
\int_0^{\pi/2}
\frac{ \sin^2(\theta) \; d\theta }
     { \sqrt{1 + x \cdot \cos^2(\theta) } },
\label{eq_big_g_def}
\end{align}
for all $x \geq 0$.
Then,
\begin{align}
H\left( \frac{n\pi}{2c} - 1 \right) < 
\frac{\chi_n - c^2}{c^2} < 
H\left( \frac{n\pi}{2c} - 1 + \frac{3\pi}{2c} \right).
\label{eq_khi_via_h}
\end{align}
Moreover,
\begin{align}
& c \cdot H\left( \frac{n\pi}{2c} - 1 \right) \cdot
          G\left( H\left( \frac{n\pi}{2c} - 1 \right) \right)
< \sqrt{\chi_n} \cdot \left(
F\left( \sqrt{\frac{\chi_n-c^2}{\chi_n}} \right) -
E\left( \sqrt{\frac{\chi_n-c^2}{\chi_n}} \right)
\right),
\label{eq_gh}
\end{align}
where $F,E$ are the complete elliptic integrals,
defined, respectively, via \eqref{eq_F}, \eqref{eq_E}
in Section~\ref{sec_elliptic}.
\label{thm_exp_term}
\end{thm}
\begin{proof}
Obviously, the function $f$, defined via \eqref{eq_f_def},
is monotonically increasing. Moreover, $f(0) = 0$, and
\begin{align}
\lim_{x \to \infty} f(x) = \infty.
\end{align}
Therefore, $H(y)$ 
is well defined for all $y \geq 0$, and, moreover,
the function $H$
is monotonically increasing. This observation,
combined with Theorems~\ref{thm_n_and_khi},
\ref{thm_n_khi_simple} in Section~\ref{sec_pswf},
implies
the inequality \eqref{eq_khi_via_h}.

Next, the right hand side of \eqref{eq_gh} increases with $\chi_n$,
due to the combination of \eqref{eq_F}, \eqref{eq_E}
in Section~\ref{sec_elliptic}. This observation,
combined with \eqref{eq_int_5} in the proof of Theorem~\ref{thm_integral},
\eqref{eq_big_g_def} and \eqref{eq_khi_via_h},
implies \eqref{eq_gh}.
\end{proof}
\begin{remark}
The functions $H,G$, defined, respectively,
via \eqref{eq_big_h_def}, \eqref{eq_big_g_def} above,
do not depend on either of $n, c, \chi_n$. Therefore,
while the right-hand side of \eqref{eq_gh} does depend on $\chi_n$,
its left-hand side depends solely on $c$ and $n$.
\end{remark}
In the following theorem, we provide simple lower and upper
bounds on $H$, defined via \eqref{eq_big_h_def}
in Theorem~\ref{thm_exp_term}.
\begin{thm}
Suppose that the function $H:[0,\infty) \to \Rc$ is defined
via \eqref{eq_big_h_def} in Theorem~\ref{thm_exp_term}.
Then,
\begin{align}
s \leq H\left( \frac{s}{4} \cdot \log \frac{16e}{s} \right)
  \leq s + \frac{s^2}{5},
\label{eq_h_bounds}
\end{align}
for all real $0 \leq s \leq 5$.
\label{thm_big_h}
\end{thm}
\begin{proof}
The proof of \eqref{eq_h_bounds} is straightforward,
elementary, and is based on
\eqref{eq_E_exp} in Section~\ref{sec_elliptic};
it will be omitted.
The correctness of Theorem~\ref{thm_big_h} has been
validated numerically.
\end{proof}
%
%
\begin{remark}
The relative error of the lower bound in \eqref{eq_h_bounds}
is below 0.07 for all $0 \leq s \leq 5$; moreover, this 
error grows roughly linearly with $s$
to $\approx 0.0085$ for all $0 \leq s \leq 0.1$.
The relative error of the upper bound in \eqref{eq_h_bounds}
grows roughly linearly with $s$ to 1, for all $0 \leq s \leq 5$.
\label{rem_h_bounds}
\end{remark}
In the following theorem, we provide simple lower and upper bound
on $G$, defined via \eqref{eq_big_g_def} in Theorem~\ref{thm_exp_term}.
\begin{thm}
Suppose that the function $G:[0,\infty) \to \Rc$ is defined
via \eqref{eq_big_g_def} in Theorem~\ref{thm_exp_term}.
Then,
\begin{align}
\frac{\pi}{4} \cdot \left(1 - \frac{x}{8}\right) \leq 
G(x) \leq
\frac{\pi}{4},
\label{eq_g_bounds}
\end{align}
for all real $0 \leq x \leq 5$.
\label{thm_big_g}
\end{thm}
\begin{proof}
The proof of \eqref{eq_g_bounds} is elementary, and is based
on the fact that, for all $x>0$,
\begin{align}
G(x) = \frac{\pi}{4} \cdot \left(1 - \frac{x}{8} + \frac{3x^2}{64} + O(x^3)
                           \right),
\end{align}
where $G$ is defined via \eqref{eq_big_g_def} in Theorem~\ref{thm_exp_term}.
The correctness of Theorem~\ref{thm_big_g} has been
validated numerically.
\end{proof}
%
%
\begin{remark}
The relative errors of both lower and upper bounds in \eqref{eq_g_bounds}
are below 0.6 for all $0 \leq x \leq 5$; moreover, these
errors are below 0.01 for all $0 \leq x \leq 0.1$, and grow
roughly linearly with $x$ in this interval.
\label{rem_g_bounds}
\end{remark}
The following theorem is in the spirit of Theorems~\ref{thm_big_h},
\ref{thm_big_g}.
\begin{thm}
Suppose that the functions $H,G:[0,\infty) \to \Rc$ are defined,
respectively,
via \eqref{eq_big_h_def}, \eqref{eq_big_g_def} in Theorem~\ref{thm_exp_term}.
Then,
\begin{align}
\frac{\pi}{4} \cdot s \cdot \left(1 - \frac{s}{8}\right) \leq 
H\left( \frac{s}{4} \cdot \log \frac{16e}{s} \right) \cdot
G\left( H\left( \frac{s}{4} \cdot \log \frac{16e}{s} \right) \right) \leq
\frac{\pi}{4} \cdot s,
\label{eq_hg_bounds}
\end{align}
for all real $0 \leq s \leq 5$. Moreover, the function
$x \to H(x) \cdot G(H(x))$ is monotonically increasing.
\label{thm_big_hg}
\end{thm}
\begin{proof}
The proof is based on Theorems~\ref{thm_big_h}, \ref{thm_big_g},
is elementary, and will be omitted.
The correctness of Theorem~\ref{thm_big_hg} has been
validated numerically.
\end{proof}
\begin{remark}
The relative errors of both lower and upper bounds in \eqref{eq_hg_bounds}
are below 0.5 for all $0 \leq s \leq 5$. Moreover, these errors
are below 0.01 for all $0 \leq s \leq 0.1$, and grow
roughly linearly with $s$ in this interval.
\label{rem_hg_bounds}
\end{remark}
The following theorem is a consequence of 
Theorems~\ref{thm_exp_term} - \ref{thm_big_hg}.
\begin{thm}
Suppose that $\delta > 0$ is a real number,
such that
\begin{align}
0 < \delta < \frac{5 \pi}{4} \cdot c.
\label{eq_delta}
\end{align}
Suppose also that $n$ is a positive integer, such that
\begin{align}
n > \frac{2}{\pi} c + \frac{2}{\pi^2} \cdot \delta \cdot
    \log\left( \frac{4e\pi c}{\delta} \right).
\label{eq_n_good}
\end{align}
Then,
\begin{align}
\delta \cdot  
\left(1 - \frac{\delta}{2\pi c}\right) < 
\sqrt{\chi_n} \cdot \left(
F\left( \sqrt{\frac{\chi_n-c^2}{\chi_n}} \right) -
E\left( \sqrt{\frac{\chi_n-c^2}{\chi_n}} \right)
\right),
\label{eq_gh_simple}
\end{align}
where $F,E$ are the complete elliptic integrals,
defined, respectively, via \eqref{eq_F}, \eqref{eq_E}
in Section~\ref{sec_elliptic}.
\label{thm_simple_exp}
\end{thm}
\begin{proof}
It follows from \eqref{eq_n_good} that
\begin{align}
\frac{\pi n}{2c} - 1 > \frac{1}{\pi} \cdot \frac{\delta}{c} \cdot
                       \log\left( \frac{4e\pi c}{\delta} \right).
\label{eq_n_good2}
\end{align}
We define the real number $s>0$ via the formula
\begin{align}
s = \frac{4\delta}{\pi c},
\label{eq_s_delta}
\end{align}
and observe that $0 < s < 5$ due to \eqref{eq_delta}.
We combine \eqref{eq_n_good2}, \eqref{eq_s_delta}
and Theorem~\ref{thm_big_hg} to obtain
\begin{align}
& H\left( \frac{n\pi}{2c} - 1 \right) \cdot
          G\left( H\left( \frac{n\pi}{2c} - 1 \right) \right) > \nonumber \\
& H\left( 
     \frac{1}{\pi} \cdot \frac{\delta}{c} \cdot
                       \log\left( \frac{4e\pi c}{\delta} \right)
 \right) \cdot
          G\left( H\left( 
\frac{1}{\pi} \cdot \frac{\delta}{c} \cdot
                       \log\left( \frac{4e\pi c}{\delta} \right)
\right) \right) = \nonumber \\
& H\left( \frac{s}{4} \cdot \log \frac{16e}{s} \right) \cdot
G\left( H\left( \frac{s}{4} \cdot \log \frac{16e}{s} \right) \right) \geq
 \frac{\pi}{4} \cdot s \cdot \left(1 - \frac{s}{8}\right)
= \frac{\delta}{c} \cdot \left(1 - \frac{\delta}{2\pi c}\right).
\label{eq_hg_2}
\end{align}
We substitute \eqref{eq_hg_2} into
the inequality \eqref{eq_gh} in Theorem~\ref{thm_exp_term}
to obtain \eqref{eq_gh_simple}.
\end{proof}
In the following theorem, we derive an upper bound on $\chi_n$
in terms of $n$.
\begin{thm}
Suppose that $n$ is a positive integer, and that
\begin{align}
\frac{2c}{\pi} < n \leq
\frac{2c}{\pi} + \frac{2}{\pi^2} \cdot \delta \cdot
    \log\left( \frac{4e\pi c}{\delta} \right) - 3,
\label{eq_khi_delta}
\end{align}
for some
\begin{align}
3 < \delta < \frac{5 \pi}{4} \cdot c.
\label{eq_delta_khi}
\end{align}
Then,
\begin{align}
\frac{\chi_n - c^2}{c^2} < \frac{8}{\pi} \cdot \frac{\delta}{c}.
\label{eq_khi_3_ineq}
\end{align}
\label{thm_khi_n_upper}
\end{thm}
\begin{proof}
We combine \eqref{eq_khi_delta}, \eqref{eq_delta_khi},
\eqref{eq_khi_via_h} in Theorem~\ref{thm_exp_term}
and \eqref{eq_h_bounds} in Theorem~\ref{thm_big_h} to obtain
\begin{align}
\frac{\chi_n - c^2}{c^2} & \; < 
H\left( \frac{(n+3) \cdot \pi}{2c} - 1 \right) <
H \left(
\frac{1}{\pi} \cdot \frac{\delta}{c} \cdot
                       \log\left( \frac{4e\pi c}{\delta} \right)
\right) < \frac{4\delta}{\pi c} \cdot 
  \left(1 + \frac{4}{5\pi} \cdot \frac{\delta}{c} \right),
\label{eq_khi_3_long}
\end{align}
which implies \eqref{eq_khi_3_ineq}. We also observe that 
\eqref{eq_delta_khi} implies
\begin{align}
\frac{2}{\pi^2} \cdot \delta \cdot
    \log\left( \frac{4e\pi c}{\delta} \right) - 3 > 1.3,
\label{eq_khi_not_empty}
\end{align}
and hence there exist integer $n$ that satisfy \eqref{eq_khi_delta}.
\end{proof}
In the following theorem, we derive an upper bound on
the non-exponential term of $\zeta(n,c)$, defined via
\eqref{eq_zeta_n_c} in Theorem~\ref{thm_big_inequality}.
\begin{thm}
Suppose that $n$ is an even positive integer, and that
\begin{align}
\frac{2c}{\pi} < n \leq
\frac{2}{\pi} c + \frac{2}{\pi^2} \cdot \delta \cdot
    \log\left( \frac{4e\pi c}{\delta} \right) - 3,
\label{eq_khi_delta2}
\end{align}
for some
\begin{align}
3 < \delta < \frac{5 \pi}{4} \cdot c.
\label{eq_delta_khi2}
\end{align}
Then,
\begin{align}
& \frac{7}{2|\psi_n(0)|} \cdot
\frac{ \left(4 \cdot \chi_n/c^2- 2 \right)^{4} }
     {       3 \cdot \chi_n/c^2 - 1            }  \cdot
\left(\chi_n-c^2\right)^{\frac{1}{4}} < \nonumber \\
& \frac{448}{3} \cdot \left(\frac{8}{\pi}\right)^{1/4} 
\cdot \delta^{1/4} \cdot c^{3/4} \cdot
\left(1 + \frac{6\delta}{\pi c}\right) \cdot
\left(1 + \frac{16\delta}{\pi c}\right)^{3}.
\label{eq_slow_term}
\end{align}
\label{thm_slow_term}
\end{thm}
\begin{proof}
We use \eqref{eq_khi_3_ineq} to obtain
\begin{align}
\frac{ \left(4 \cdot \chi_n/c^2- 2 \right)^{4} }
     {       3 \cdot \chi_n/c^2 - 1            } & \; =
\frac{4}{3} \cdot
\frac{\left(4 \cdot (\chi_n-c^2)/c^2 + 2\right)^4 }
     { 4 \cdot (\chi_n-c^2)/c^2 + 8/3 } \nonumber \\
& \; < \frac{32}{3} \cdot \left(1 + 2 \cdot \frac{\chi_n-c^2}{c^2} \right)^3
 < \frac{32}{3} \cdot \left( 1 + \frac{16 \delta}{ \pi c} \right)^3.
\label{eq_slow_1}
\end{align}
Then, we use \eqref{eq_khi_3_ineq} to obtain
\begin{align}
\left(\chi_n-c^2\right)^{\frac{1}{4}} < 
\left( \frac{8 \delta c}{\pi} \right)^{\frac{1}{4}}.
\label{eq_slow_2}
\end{align}
Next, we combine Theorems~\ref{thm_n_khi_simple}, \ref{thm_psi0}
in Section~\ref{sec_pswf} with Theorem~\ref{thm_khi_n_upper} to obtain
\begin{align}
\frac{1}{| \psi_n(0) |} 
& \; < 4 \cdot \sqrt{n} \cdot \sqrt{\frac{\chi_n}{c^2}} 
< \frac{4}{c} \cdot \left( \chi_n \right)^{\frac{3}{4}}
< \frac{4}{c} \cdot c^{\frac{3}{2}} \cdot 
\left(1 + \frac{8\delta}{\pi c}\right)^{\frac{3}{4}}  <
4 \cdot c^{\frac{1}{2}} \cdot \left(1 + \frac{6\delta}{\pi c}\right).
\label{eq_slow_3}
\end{align}
We combine \eqref{eq_slow_1}, \eqref{eq_slow_2}, \eqref{eq_slow_3}
to obtain \eqref{eq_slow_term}.
\end{proof}
The following theorem is one of the principal results
of this subsection.
\begin{thm}
Suppose that $c>0$ is a real number, and that
\begin{align}
c>22.
\label{eq_crude_c22}
\end{align}
Suppose also that $\delta>0$ is a real number, and that
\begin{align}
3 < \delta < \frac{\pi c}{16}.
\label{eq_delta_crude}
\end{align}
Suppose, in addition, that $n$ is a positive integer, and that
\begin{align}
n \geq \frac{2c}{\pi}  + \frac{2}{\pi^2} \cdot \delta \cdot
    \log\left( \frac{4e\pi c}{\delta} \right).
\label{eq_n_crude}
\end{align}
Suppose furthermore that the real number $\xi(n,c)$ is defined
via the formula
\begin{align}
\xi(n,c) = 7056 \cdot c \cdot 
\exp\left[-\delta\left(1 - \frac{\delta}{2\pi c}\right) \right].
\label{eq_xi_n_c}
\end{align}
Then, 
\begin{align}
| \lambda_n | < \xi(n,c).
\label{eq_crude_inequality}
\end{align}
\label{thm_crude_inequality}
\end{thm}
\begin{proof}
Suppose first that $n$ is an even positive integer of the form
\begin{align}
n = \frac{2c}{\pi} + \frac{2}{\pi^2} \cdot \delta \cdot
    \log\left( \frac{4e\pi c}{\delta} \right),
\label{eq_crude_two_deltas}
\end{align}
for some $3 < \delta < \pi c/16$ (in other words,
\eqref{eq_n_crude} is an identity rather than an inequality).
We observe that, for all real $t > 0$, 
\begin{align}
\frac{d}{dt} \left(
t \cdot \log \left( \frac{4e\pi c}{t} \right)
\right) = \log\left( \frac{4\pi c}{t} \right).
\label{eq_crude_der_log}
\end{align}
We combine \eqref{eq_crude_c22} with \eqref{eq_crude_der_log} to obtain
\begin{align}
\frac{2}{\pi^2} \cdot \left(
\frac{\pi c}{8} \cdot \log\left( \frac{4e\pi c}{(\pi c)/8} \right) -
\delta \cdot \log\left( \frac{4e\pi c}{\delta} \right) 
\right) & \; >
\frac{2}{\pi^2} \cdot
\left(\frac{\pi c}{8}-\delta\right) \cdot 
\log\left(\frac{4\pi c}{(\pi c)/8}\right) \nonumber \\
& \;  > \frac{c}{8\pi} \cdot \log \left( 32 \right) > 3.
\end{align}
Therefore, it is possible to choose a real number $\hat{\delta}$
such that
\begin{align}
3 < \hat{\delta} <  \frac{\pi c}{8},
\label{eq_delta_hat_range}
\end{align}
and also
\begin{align}
n = \frac{2c}{\pi} + \frac{2}{\pi^2} \cdot \hat{\delta} \cdot
    \log\left( \frac{4e\pi c}{\hat{\delta}} \right) - 3.
\label{eq_crude_delta_hat}
\end{align}
Due to the combination of \eqref{eq_delta_hat_range},
\eqref{eq_crude_delta_hat} and Theorem~\ref{thm_slow_term},
\begin{align}
& \frac{7}{2|\psi_n(0)|} \cdot
\frac{ \left(4 \cdot \chi_n/c^2- 2 \right)^{4} }
     {       3 \cdot \chi_n/c^2 - 1            }  \cdot
\left(\chi_n-c^2\right)^{\frac{1}{4}} < \nonumber \\
& \frac{448}{3} \cdot c^{1/4} \cdot c^{3/4} \cdot
\left(1 + \frac{6}{8}\right) \cdot
\left(1 + \frac{32}{16}\right)^{3} = 7056 \cdot c.
\label{eq_slow_term_c}
\end{align}
We observe that the right-hand side of \eqref{eq_slow_term_c}
is independent of $\hat{\delta}$. We combine 
this observation with \eqref{eq_slow_term_c},
\eqref{eq_big_inequality} in Theorem~\ref{thm_big_inequality},
\eqref{eq_gh_simple} in Theorem~\ref{thm_simple_exp},
and the fact that $|\lambda_n|$ decrease monotonically with $n$,
to obtain \eqref{eq_crude_inequality}.
\end{proof}
\begin{definition}[$\delta(n)$]
Suppose that $n$ is a positive integer, and that
\begin{align}
\frac{2c}{\pi} < n < \frac{10c}{\pi}.
\label{eq_range_for_delta}
\end{align}
We define the real number $\delta(n)$ to be the solution of the equation
\begin{align}
n = \frac{2c}{\pi} + \frac{2}{\pi^2} \cdot X \cdot
    \log\left( \frac{4e\pi c}{X} \right),
\label{eq_def_delta_n}
\end{align}
in the unknown $X$ in the interval $0 < X < 4 \pi c$.
\label{def_delta_n}
\end{definition}
\begin{remark}
We observe that the right-hand side of \eqref{eq_def_delta_n}
is an increasing function of $X$ in the range
$0 < X < 4\pi c$,
due to \eqref{eq_crude_der_log}
in the proof of Theorem~\ref{thm_crude_inequality}. 
Therefore, $\delta(n)$ is well defined.
\label{rem_crude}
\end{remark}
In the following theorem, we derive yet another upper bound
on $|\lambda_n|$.
\begin{thm}
Suppose that $n>0$ is a positive integer, and that $n > (2c/\pi) + \sqrt{42}$.
Suppose also that the real number $x_n$ is defined via the formula
\begin{align}
x_n = \frac{\chi_n}{c^2}.
\label{eq_lambda_khi_1}
\end{align}
Then,
\begin{align}
|\lambda_n| < 
1195 \cdot c \cdot (x_n)^{\frac{3}{4}} \cdot (x_n-1)^{\frac{1}{4}} \cdot
\left(x_n-\frac{1}{2}\right)^3 \cdot
\exp\left[ -\frac{\pi}{4} \cdot \left(\sqrt{x_n}-\frac{1}{\sqrt{x_n}} \right)
\cdot c \right].
\label{eq_lambda_khi}
\end{align}
\label{thm_lambda_khi}
\end{thm}
\begin{proof}
We use \eqref{eq_lambda_khi_1} to obtain
\begin{align}
\frac{ \left(4 \cdot \chi_n/c^2- 2 \right)^{4} }
     {       3 \cdot \chi_n/c^2 - 1            } & \; =
\frac{4}{3} \cdot
\frac{\left(4 \cdot (\chi_n-c^2)/c^2 + 2\right)^4 }
     { 4 \cdot (\chi_n-c^2)/c^2 + 8/3 }
< \frac{256}{3} \cdot \left(x_n - \frac{1}{2} \right)^3.
\label{eq_lambda_khi_a}
\end{align}
Next, we combine Theorems~\ref{thm_n_khi_simple}, \ref{thm_psi0}
in Section~\ref{sec_pswf}
and \eqref{eq_lambda_khi_1} to obtain
\begin{align}
\frac{1}{| \psi_n(0) |} 
& \; < 4 \cdot \sqrt{n} \cdot \sqrt{\frac{\chi_n}{c^2}} 
< \frac{4}{c} \cdot \left( \chi_n \right)^{\frac{3}{4}} =
  4 \sqrt{c} \cdot (x_n)^{\frac{3}{4}}.
\label{eq_lambda_khi_b}
\end{align}
We combine \eqref{eq_lambda_khi_a} and \eqref{eq_lambda_khi_b} to obtain
\begin{align}
\frac{7}{2} \cdot 
\frac{\left(\chi_n-c^2\right)^{\frac{1}{4}}}{| \psi_n(0) |} \cdot 
\frac{ \left(4 \cdot \chi_n/c^2- 2 \right)^{4} }
     {       3 \cdot \chi_n/c^2 - 1            } <
1195 \cdot c \cdot (x_n)^{\frac{3}{4}} \cdot (x_n-1)^{\frac{1}{4}} \cdot
\left(x_n-\frac{1}{2}\right)^3.
\label{eq_lambda_khi_c}
\end{align}
Also, we combine 
\eqref{eq_F}, \eqref{eq_E} in Section~\ref{sec_elliptic} with
\eqref{eq_fme} in the proof of Theorem~\ref{thm_integral} to obtain
\begin{align}
F\left( \sqrt{\frac{\chi_n-c^2}{\chi_n}} \right) -
E\left( \sqrt{\frac{\chi_n-c^2}{\chi_n}} \right)
> \frac{\pi}{4} \cdot \frac{\chi_n-c^2}{\chi_n}.
\label{eq_lambda_khi_d}
\end{align}
We combine \eqref{eq_lambda_khi_1},
\eqref{eq_lambda_khi_c}, \eqref{eq_lambda_khi_d}
with Theorems~\ref{thm_khi_2},~\ref{thm_big_inequality} to obtain
\eqref{eq_lambda_khi}.
\end{proof}
We conclude this subsection with the following theorem,
that describes the behavior of the upper bound $\nu(n,c)$
on $|\lambda_n|$ (see \eqref{eq_nu}, \eqref{eq_lambda_nu}
in Theorem~\ref{thm_rokhlin} in Section~\ref{sec_pswf}).
\begin{thm}
Suppose that
$n$ is a positive integer, and that
\begin{align}
\frac{2}{\pi} \cdot c \leq 
n < \left( \frac{2}{\pi} + \frac{1}{25} \right) \cdot c.
\label{eq_n_nu}
\end{align}
Then, 
\begin{align}
\nu(n,c) \geq \frac{1}{10},
\label{eq_nu_gc_1}
\end{align}
where $\nu(n,c)$ is defined via \eqref{eq_nu} in
Theorem~\ref{thm_rokhlin} in Section~\ref{sec_pswf}.
\label{thm_nu}
\end{thm}
\begin{proof}
We carry out elementary calculations, involving the
well known Stirling's approximation formula for the gamma function,
to obtain the inequality
\begin{align}
\nu(n,c) \geq \frac{\sqrt{2\pi n}}{2n+1} \cdot 
              \left( \frac{ce}{4n} \right)^n,
\label{eq_nu_1}
\end{align}
for all $n$ in the range \eqref{eq_n_nu}.
We use \eqref{eq_nu_1} to obtain the inequality
\begin{align}
\log (\nu(n,c)) & \; > \log \frac{1}{\sqrt{n}} + 
  n \cdot \log \left(  \frac{ce}{4n} \right) \nonumber \\
 & \; > -\frac{1}{2} \cdot \log(c) +
   \left( \frac{2}{\pi} + \frac{1}{25} \right) \cdot c \cdot
   \log\left( \frac{e/4}{2/\pi + 1/25} \right) \nonumber \\
 & \; > -\frac{1}{2} \cdot \log(c) + \frac{c}{500} 
 \geq \frac{1}{2} \cdot \left(1-\log(250)\right) > -2.27.
\label{eq_nu_2}
\end{align}
The inequality \eqref{eq_nu_gc_1} follows directly from \eqref{eq_nu_2}.
\end{proof}
\begin{remark}
According to Theorem~\ref{thm_nu}, the inequality \eqref{eq_lambda_nu}
of Theorem~\ref{thm_rokhlin} in Section~\ref{sec_pswf} is trivial
for all integer $n < (2/\pi+1/25)\cdot c$. In particular, for such $n$
this inequality is useless.
\label{rem_nu}
\end{remark}

\section{Numerical Results}
\label{sec_numerical}
In this section,
we illustrate the results of Section~\ref{sec_analytical}
via several numerical experiments. All the calculations were
implemented in FORTRAN (the Lahey 95 LINUX version) and were
carried out in either
double or quadruple precision.
The algorithms for the evaluation of PSWFs and
the associated eigenvalues were based on \cite{RokhlinXiaoProlate}.


\subsection{Experiment 1}
\label{sec_exp1}
In this experiment, we demonstrate the behavior of $|\lambda_n|$
with $0 \leq n \leq 2c/\pi$, for several values of band limit $c>0$.

For each of five different values of $c=10,10^2,10^3,10^4,10^5$, 
we do the following. First, we evaluate 
$| \lambda_n |$ numerically, 
for $n = 0$, $n \approx c/\pi$ and $n \approx 2c/\pi$.
For each such $n$, we also compute
$\mu_n = (c/2\pi) \cdot | \lambda_n |$. Here $\lambda_n$ is
the $n$th eigenvalue of the integral operator $F_c$, and $\mu_n$ is
the $n$th eigenvalue of the integral operator $Q_c$
(see \eqref{eq_pswf_fc}, \eqref{eq_prolate_integral},
\eqref{eq_pswf_qc}, \eqref{eq_prolate_mu} 
in Section~\ref{sec_pswf}).
\begin{table}[htbp]
\begin{center}
\begin{tabular}{c|c|c|c|c}
$c$   &
$n$   &
$(\pi n) / (2c)$ &
$\left| \lambda_n \right|$ &
$ \mu_n = (c/2\pi) \cdot \left| \lambda_n \right|^2$ \\
\hline
    10 &      0 & 0.00000E+00 & 0.79267E+00 & 0.10000E+01  \\
    10 &      3 & 0.47124E+00 & 0.79183E+00 & 0.99790E+00  \\
    10 &      6 & 0.94248E+00 & 0.52588E+00 & 0.44015E+00  \\
\hline
   100 &      0 & 0.00000E+00 & 0.25066E+00 & 0.10000E+01  \\
   100 &     31 & 0.48695E+00 & 0.25066E+00 & 0.10000E+01  \\
   100 &     63 & 0.98960E+00 & 0.18589E+00 & 0.54997E+00  \\
\hline
  1000 &      0 & 0.00000E+00 & 0.79267E-01 & 0.10000E+01  \\
  1000 &    318 & 0.49951E+00 & 0.79267E-01 & 0.10000E+01  \\
  1000 &    636 & 0.99903E+00 & 0.57640E-01 & 0.52877E+00  \\
\hline
 10000 &      0 & 0.00000E+00 & 0.25066E-01 & 0.10000E+01  \\
 10000 &   3183 & 0.49998E+00 & 0.25066E-01 & 0.10000E+01  \\
 10000 &   6366 & 0.99997E+00 & 0.16644E-01 & 0.44088E+00  \\
\hline
100000 &      0 & 0.00000E+00 & 0.79267E-02 & 0.10000E+01  \\
100000 &  31830 & 0.49998E+00 & 0.79267E-02 & 0.10000E+01  \\
100000 &  63661 & 0.99998E+00 & 0.60295E-02 & 0.57861E+00  \\
\end{tabular}
\end{center}
\caption{\it
Behavior of $\left| \lambda_n \right|$ for $0 \leq n \leq 2c/\pi$.
Corresponds to Experiment 1 in Section~\ref{sec_numerical}.
}
\label{t:test152}
\end{table}

In addition, we fix $c = 100$, and evaluate $|\lambda_n|$ numerically,
for all integer $n$ between $0$ and $2c/\pi$.

The results of Experiment 1 are shown in 
Table~\ref{t:test152} 
and Figure~\ref{fig:test150b}.
Table~\ref{t:test152} 
has the following structure.
The first two columns contain the band limit $c$ and the
prolate index $n$, respectively. The third column contains
the ratio of $n$ to $2c/\pi$. The fourth column contains $| \lambda_n|$.
The last column contains the eigenvalue $\mu_n$ of the integral operator
$Q_c$ (see \eqref{eq_pswf_qc}, \eqref{eq_prolate_mu} 
in Section~\ref{sec_pswf}).

In Figure~\ref{fig:test150b}, we plot $|\lambda_n|$,
corresponding to $c=100$, as a function of $n$,
for integer $n$ between $0$ and $2c/\pi$.

Several observations can be made from Table~\ref{t:test152}
and Figure~\ref{fig:test150b}.
\begin{itemize}
\item[{\bf 1.}] For all five values of band limit $c$, 
the eigenvalue $\mu_n$ decreases from $\approx 1$ to
$\approx 1/2$,
as $n$ increases from $0$
to $(2c/\pi)$. In other words, the first $2c/\pi$ eigenvalues $\lambda_n$
have roughly the same magnitude $ \approx \sqrt{2\pi/c}$.
This observation confirms Theorem~\ref{thm_mu_spectrum}
in Section~\ref{sec_pswf}.
\item[{\bf 2.}] Due to Theorem~\ref{thm_n_and_khi}
in Section~\ref{sec_pswf},
the bounds on the decay of $| \lambda_n |$,
established in
Section~\ref{sec_analytical},
hold for
$n$ greater than $2c/\pi$ only
(see also Remark~\ref{rmk_42}). Thus, Table~\ref{t:test152} 
indicates
that this assumption on $n$ is, 
in fact, not restrictive, since the first
$2c/\pi$ eigenvalues have roughly constant magnitude.
\end{itemize}

\subsection{Experiment 2}
\label{sec_exp2}
In this experiment, we illustrate Theorem~\ref{thm_big_inequality}.
As opposed to Experiment 1, we demonstrate the behavior of $|\lambda_n|$
for $n > 2c/\pi$.

In this experiment, we proceed as follows.
First, we pick band limit $c>0$ (more or less arbitrarily). 
Then, for each even integer $n$
in the range
\begin{align}
\frac{2c}{\pi} < n < \frac{2c}{\pi} + 20 \cdot \log(c),
\label{eq_n_range}
\end{align}
we evaluate numerically $| \lambda_n |$ and $\zeta(n,c)$,
where the latter is defined via \eqref{eq_zeta_n_c}
in Theorem~\ref{thm_big_inequality}. 

The results of Experiment 2 are shown in
Figures~\ref{fig:test151_c10} - \ref{fig:test151_2} and
in Table~\ref{t:test153}.
In Figures~\ref{fig:test151_c10} - \ref{fig:test151_2},
we plot both $\log(|\lambda_n|)$
and $\log(\zeta(n,c))$ as functions of $n$. 
Each of Figures~\ref{fig:test151_c10} - \ref{fig:test151_2}
corresponds to a certain value of band limit 
($c = 10, 10^2, 10^3, 10^4, 10^5$, respectively).

\begin{table}[htbp]
\begin{center}
\begin{tabular}{c|c|c|c|c|c|c}
$\varepsilon$ &
$c$   &
$n_1(\varepsilon)$ &
$\Delta_1(\varepsilon)$ &
$n_2(\varepsilon)$ &
$\Delta_2(\varepsilon)$ &
$n_2(\varepsilon) - n_1(\varepsilon)$ \\
\hline
$e^{-50}$ & 10     & 32 & 0.11133E+02 &     38 & 0.13738E+02 &      6  \\
$e^{-50}$ & $10^2$ & 107 & 0.94107E+01 &    114 & 0.10931E+02 &      7  \\
$e^{-50}$ & $10^3$ & 700 & 0.91752E+01 &    712 & 0.10912E+02 &     12  \\
$e^{-50}$ & $10^4$ & 6450 & 0.90987E+01 &   6468 & 0.11053E+02 &     18  \\
$e^{-50}$ & $10^5$ & 63765 & 0.89484E+01 &  63792 & 0.11294E+02 &     27  \\
\hline
$e^{-100}$ & 10     & 50 & 0.18950E+02 &     56 & 0.21556E+02 &      6  \\
$e^{-100}$ & $10^2$ & 138 & 0.16142E+02 &    146 & 0.17879E+02 &      8  \\
$e^{-100}$ & $10^3$ & 753 & 0.16848E+02 &    764 & 0.18440E+02 &     11  \\
$e^{-100}$ & $10^4$ & 6526 & 0.17350E+02 &   6542 & 0.19087E+02 &     16  \\
$e^{-100}$ & $10^5$ & 63864 & 0.17547E+02 &  63890 & 0.19806E+02 &     26  \\
\end{tabular}
\end{center}
\caption{\it
Illustration of Theorem~\ref{thm_big_inequality}. Corresponds
to Experiment 2 in Section~\ref{sec_numerical}.
}
\label{t:test153}
\end{table}
Table~\ref{t:test153} 
has the following structure. The first
column contains precision $\varepsilon = e^{-50}, e^{-100}$.
The second column contains band limit $c$. The third column contains
the integer $n_1(\varepsilon)$, defined via the formula
\begin{align}
n_1(\varepsilon) = \min_k\left\{ k > 2c/\pi \; : \; 
                                 |\lambda_k| < \varepsilon \right\}.
\label{eq_n1}
\end{align}
In other words, $n_1(\varepsilon)$ is the integer satisfying the inequality
\begin{align}
|\lambda_{n_1(\varepsilon)-1}| > \varepsilon >
|\lambda_{n_1(\varepsilon)}|.
\end{align}
The fourth column contains
$\Delta_1(\varepsilon)$, defined to be the difference
between $n_1(\varepsilon)$
and $2c/\pi$, scaled by $\log(c)$. In other words,
\begin{align}
\Delta_1(\varepsilon) = \frac{n_1(\varepsilon) - 2c/\pi}{\log(c)}.
\label{eq_delta_1}
\end{align}
The fifth column contains
the even integer $n_2(\varepsilon)$, defined via the formula
\begin{align}
n_2(\varepsilon) = \min_k\left\{ k > 2c/\pi \; : \; k \text{ is even}, \;
                                 |\zeta(k,c)| < \varepsilon \right\}.
\label{eq_n2}
\end{align}
In other words, $n_2(\varepsilon)$ is the even integer satisfying the inequality
\begin{align}
|\zeta(n_2(\varepsilon)-2,c)| > \varepsilon >
|\zeta(n_2(\varepsilon),c)|.
\end{align}
The sixth column contains 
$\Delta_2(\varepsilon)$, defined to be the difference
between $n_2(\varepsilon)$
and $2c/\pi$, scaled by $\log(c)$. In other words,
\begin{align}
\Delta_2(\varepsilon) = \frac{n_2(\varepsilon) - 2c/\pi}{\log(c)}.
\label{eq_delta_2}
\end{align}
The last column contains the difference between $n_2(\varepsilon)$
and $n_1(\varepsilon)$.

Several observations can be made
from 
Figures~\ref{fig:test151_c10} - \ref{fig:test151_2} and
Table~\ref{t:test153}.

\begin{itemize}
\item[{\bf 1.}] In all figures, $|\lambda_n| < \zeta(n,c)$,
as expected, which confirms
Theorem~\ref{thm_big_inequality}. 
\item[{\bf 2.}] For each $c$, both $|\lambda_n|$ and
$\zeta(n,c)$ decay roughly exponentially fast with $n$. 
\item[{\bf 3.}] For each $c$, both $|\lambda_n|$ and
$\zeta(n,c)$ decrease to roughly $e^{-125}$, as $n$
increases from $2c/\pi$ to $2c/\pi+20 \cdot \log(c)$.
In particular,
\begin{align}
\left| \lambda_{2c/\pi + 20 \cdot \log(c)} \right| \approx e^{-125},
\label{eq_lambda_20}
\end{align}
for $c = 10, 10^2, 10^3, 10^4, 10^5$. 
The fact that the right-hand side of \eqref{eq_lambda_20}
is the same for all $c$ is somewhat surprising. However,
this is not coincidental, as will be illustrated in
Experiment 3 below.
\item[{\bf 4.}] For $c=10^2,10^3,10^4,10^5$, it suffices to take
$n \approx 2c/\pi + 9 \cdot \log(c)$ to ensure that 
$|\lambda_n| \approx e^{-50}$ (see third column in Table~\ref{t:test153}).
In addition, it suffices to take 
$n \approx 2c/\pi + 17 \cdot \log(c)$ to ensure that 
$|\lambda_n| \approx e^{-100}$. In other words, 
\begin{align}
n_1(\varepsilon) \approx 
\frac{2c}{\pi} + 0.17 \cdot \log\left(\frac{1}{\varepsilon}\right)
                     \cdot \log(c),
\label{eq_n1_04}
\end{align}
where $n_1(\varepsilon)$ is defined via \eqref{eq_n1} above
(see also \eqref{eq_lambda_20}).
\item[{\bf 5.}]
For $c=10^2,10^3,10^4,10^5$, it suffices to take
$n \approx 2c/\pi + 11 \cdot \log(c)$ to ensure that 
$\zeta(n,c) \approx e^{-50}$ (see fifth column in Table~\ref{t:test153}).
In addition, it suffices to take 
$n \approx 2c/\pi + 19 \cdot \log(c)$ to ensure that 
$\zeta(n,c) \approx e^{-100}$. In other words, 
\begin{align}
n_2(\varepsilon) \approx 
\frac{2c}{\pi} + 0.2 \cdot \log\left(\frac{1}{\varepsilon}\right)
                     \cdot \log(c),
\label{eq_n2_05}
\end{align}
where $n_2(\varepsilon)$ is defined via \eqref{eq_n2} above
(see also \eqref{eq_lambda_20}, \eqref{eq_n1_04}).
\item[{\bf 6.}] The difference $n_2(\varepsilon)-n_1(\varepsilon)$
is roughly independent of $\varepsilon$, and grows only
slowly as $c$ increases (see last column of Table~\ref{t:test153}).
In other words, suppose that one needs to determine $n$ such
that $|\lambda_k|<e^{-50}$ for all $k \geq n$. 
Due to \eqref{eq_n1}, $n_1(e^{-50})$ would be the minimal such $n$. 
On the other hand, $n=n_2(e^{-50})$ is only larger by 6 for $c=10$
and by 27 for $c=10^5$.
\end{itemize}

\subsection{Experiment 3}
\label{sec_exp3}
In this experiment, we illustrate Theorem~\ref{thm_crude_inequality}.
We proceed as follows.
First, we pick band limit $c>0$ (more or less arbitrarily). 
Then, we define the positive integer $n_{\max}$ to be the minimal
even integer such that
\begin{align}
n_{\max} >
\frac{2c}{\pi}  + \frac{2}{\pi^2} \cdot 150 \cdot
    \log\left( \frac{4e\pi c}{150} \right)
\approx \frac{2c}{\pi} + 30.4 \cdot \log( 0.23 \cdot c).
\label{eq_nmax_exp3}
\end{align}
Then, for each positive even integer $n$ in the range
\begin{align}
\frac{2c}{\pi} < n < n_{\max},
\label{eq_n_range_exp3}
\end{align}
we evaluate the following quantities:
\begin{itemize}
\item the eigenvalue $\lambda_n$ of the operator $F_c$ (see
\eqref{eq_pswf_fc}, \eqref{eq_prolate_integral}
in Section~\ref{sec_pswf});
\item $\delta(n)$ of Definition~\ref{def_delta_n} in 
Section~\ref{sec_weaker};
\item $\zeta(n,c)$, defined via
\eqref{eq_zeta_n_c} in Theorem~\ref{thm_big_inequality}
in Section~\ref{sec_principal};
\item $\xi(n,c)$, defined via \eqref{eq_xi_n_c}
in Theorem~\ref{thm_crude_inequality} in Section~\ref{sec_weaker}.
\end{itemize}
The results of Experiment 3 are shown in
Figures~\ref{fig:test154_c10000}, \ref{fig:test154_c100000}, that
correspond, respectively, to band limit $c=10^4$ and $c=10^5$.
In each of 
Figures~\ref{fig:test154_c10000}, \ref{fig:test154_c100000}, 
we plot $\log(|\lambda_n|)$, $-\delta(n)$,
$\log(\zeta(n,c))$ and $\log(\xi(n,c))$ as functions of $n$.

Several observations can be made from
Figures~\ref{fig:test154_c10000}, \ref{fig:test154_c100000}, and
from more detailed experiments by the author.
\begin{itemize}
\item[{\bf 1.}] In both figures, 
\begin{align}
\log(|\lambda_n|) < -\delta(n) < \log(\zeta(n,c)) < \log(\xi(n,c)),
\end{align}
for all $n$. This observation confirms both
Theorem~\ref{thm_big_inequality} of Section~\ref{sec_principal}
and Theorem~\ref{thm_crude_inequality} of Section~\ref{sec_weaker}.
Also, $\xi(n,c)$ is weaker than $\zeta(n,c)$ as an upper bound
on $|\lambda_n|$, as expected.
\item[{\bf 2.}] All the four functions, plotted
in Figures~\ref{fig:test154_c10000}, \ref{fig:test154_c100000},
decay roughly exponentially with $n$. Moreover, 
\begin{align}
\log(|\lambda_n|) \approx \log\sqrt{ \frac{2\pi}{c} } - \delta(n),
\end{align}
in correspondence with Theorem~\ref{thm_approx} in Section~\ref{sec_pswf}.
In particular, even the weakest bound $\xi(n,c)$ correctly captures
the exponential decay of $|\lambda_n|$. On the other hand,
$\xi(n,c)$ overestimates $|\lambda_n|$ by a 
roughly constant factor of order $c^{3/2}$ (see also Section~\ref{sec_inac}).
\end{itemize}



\clearpage

\begin{figure} [t]
\begin{center}
\includegraphics[width=11.5cm, bb=81 217 549 564, clip=true]
{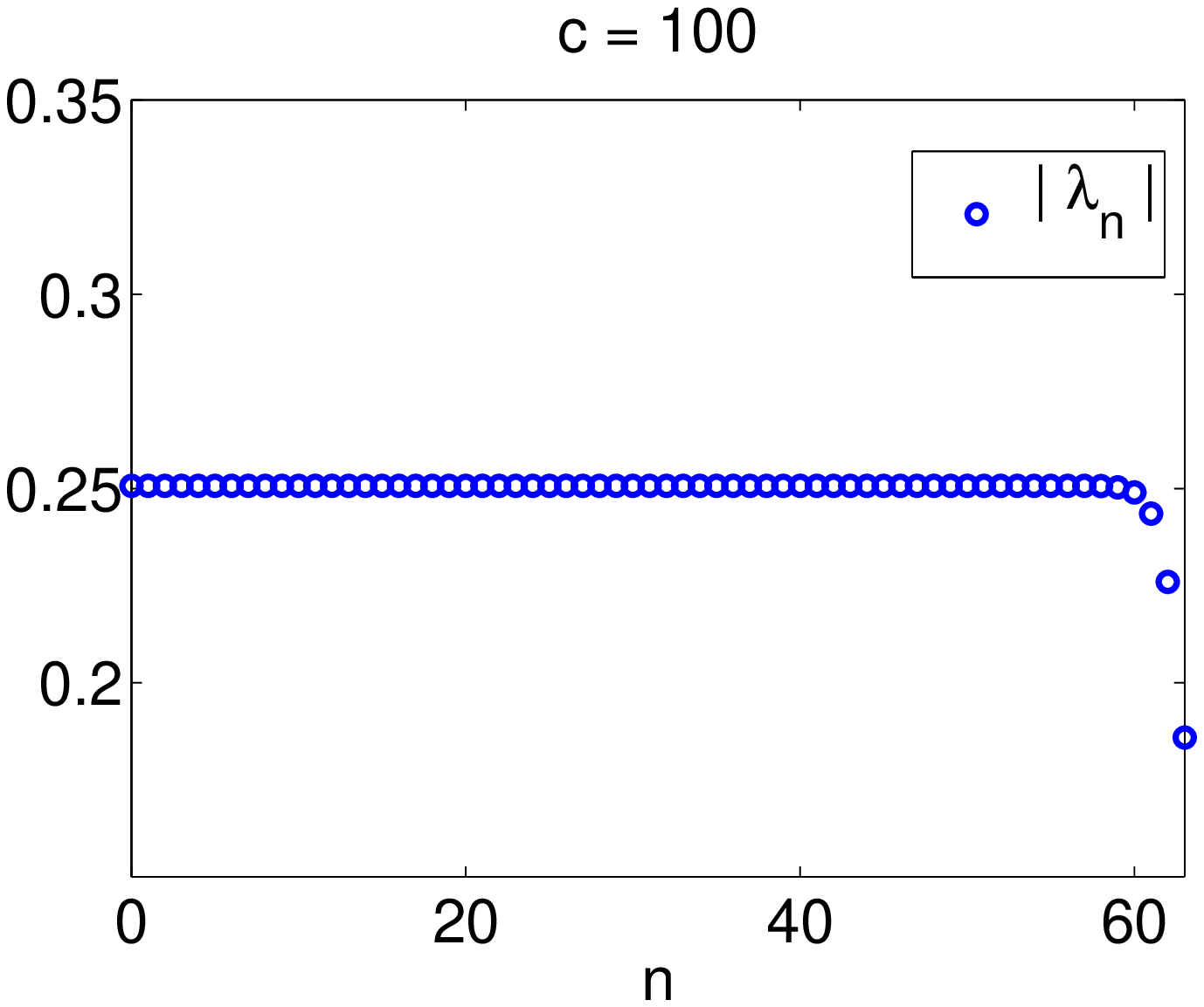}
\caption
{
Behavior of $|\lambda_n|$ for $0 < n < 2c/\pi$, with $c=100$.
Corresponds to Experiment 1 in Section~\ref{sec_numerical}.
}
\label{fig:test150b}
\end{center}
\end{figure}
\begin{figure} [h]
\begin{center}
\includegraphics[width=11.5cm, bb=81 217 549 564, clip=true]
{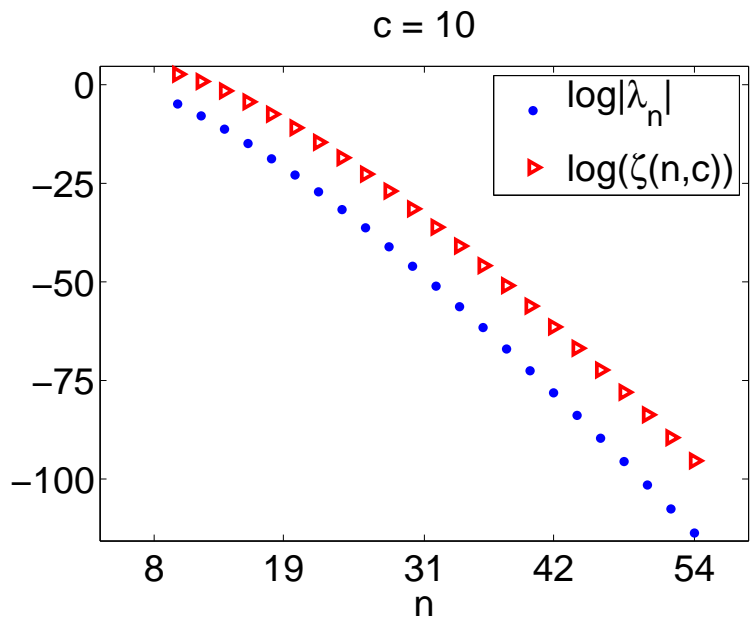}
\caption
{
Illustration of Theorem~\ref{thm_big_inequality} with
$c=10$. Corresponds to Experiment 2 in Section~\ref{sec_numerical}.
}
\label{fig:test151_c10}
\end{center}
\end{figure}
\clearpage

\begin{figure}[ht]
\centering
\subfigure[$c=100$.]{
\includegraphics[width=11.5cm, bb=81 217 549 564, clip=true]
{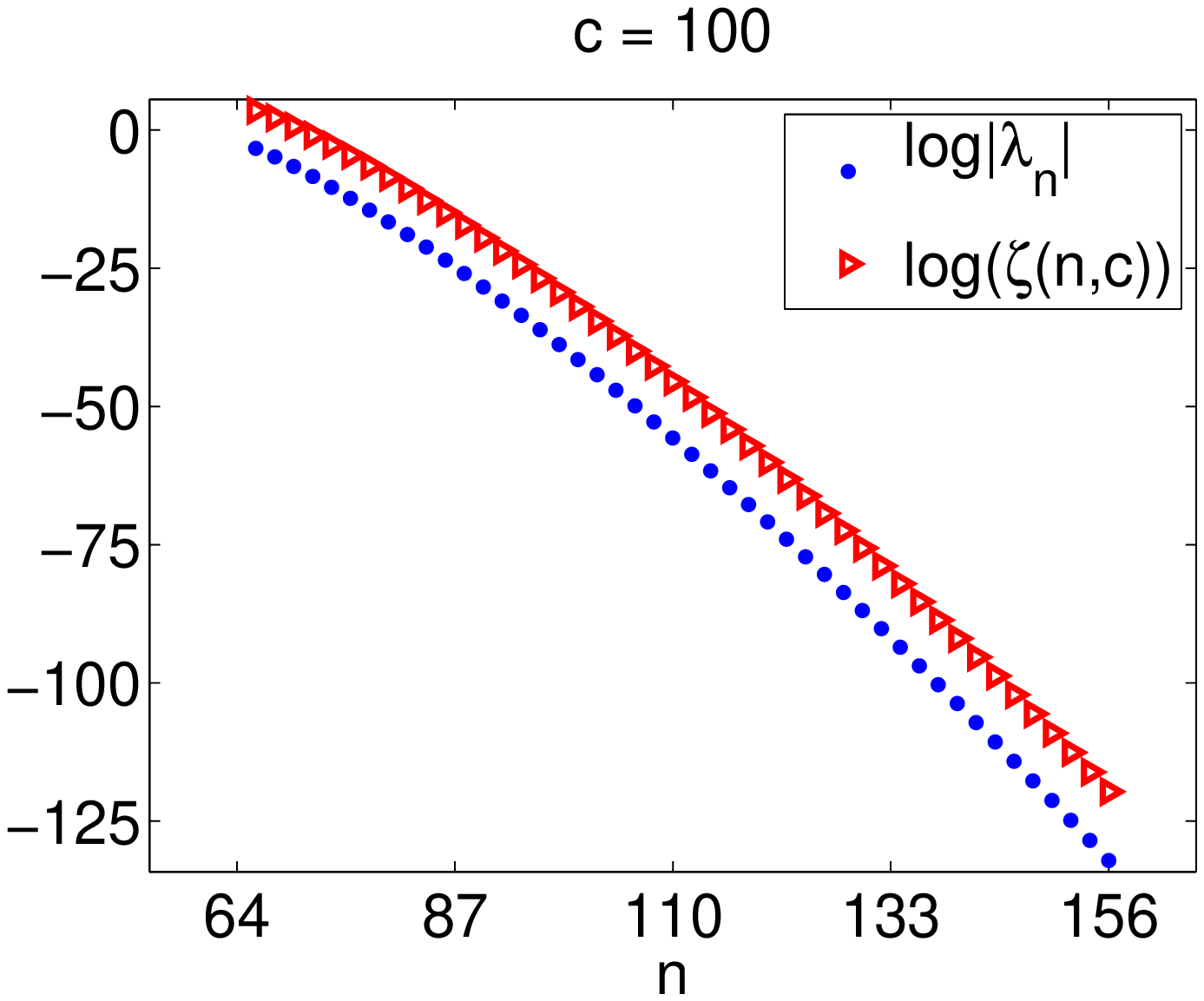}
\label{fig:test151_c100}
}
\subfigure[$c=1,000$.]{
\includegraphics[width=11.5cm, bb=81 217 549 564, clip=true]
{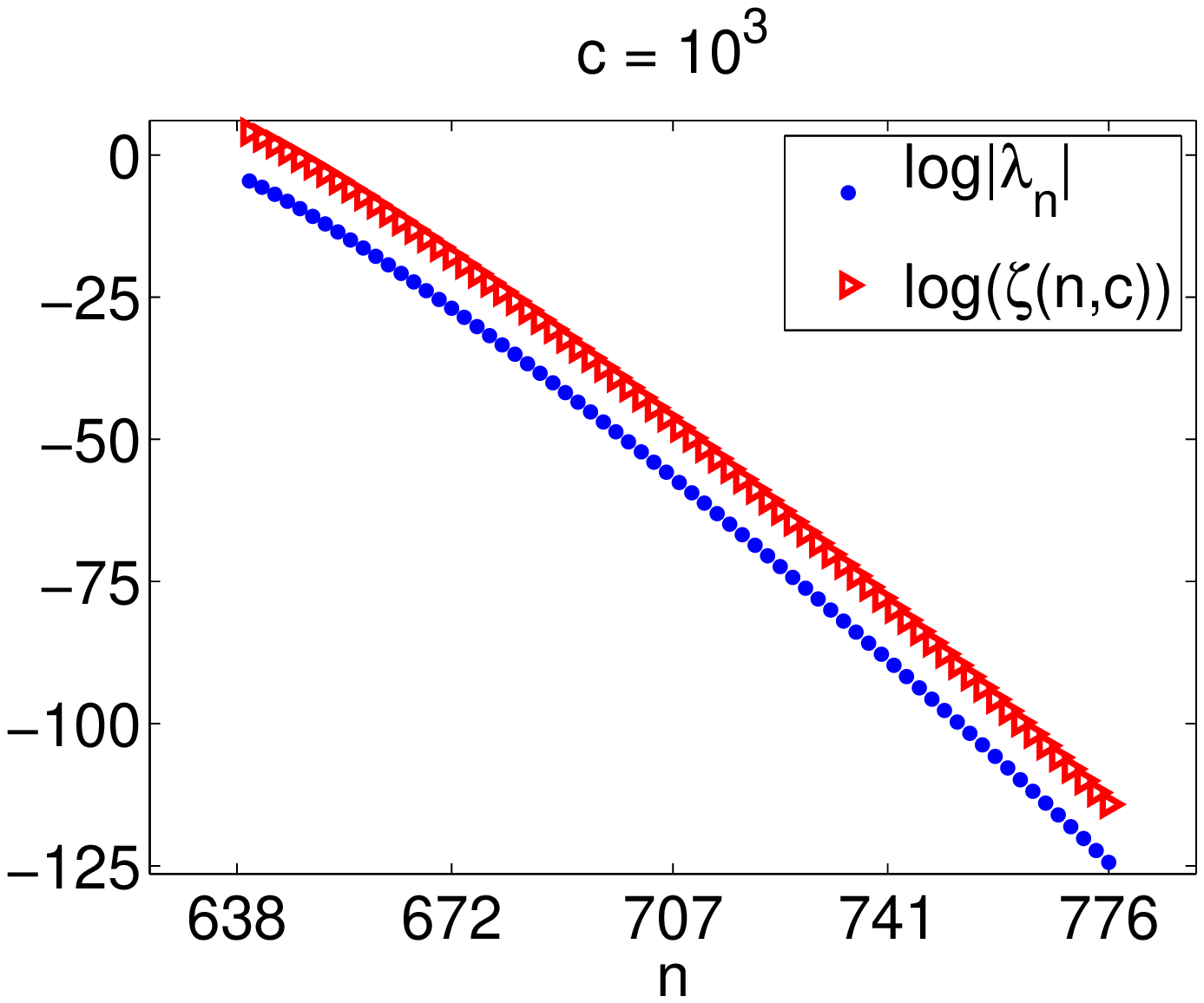}
\label{fig:test151_c1000}
}
\caption{ 
Illustration of Theorem~\ref{thm_big_inequality}. 
Corresponds to Experiment 2 in Section~\ref{sec_numerical}.
}
\label{fig:test151_1}
\end{figure}

\clearpage

\begin{figure}[ht]
\centering
\subfigure[$c=10,000$.]{
\includegraphics[width=11.5cm, bb=81 217 549 564, clip=true]
{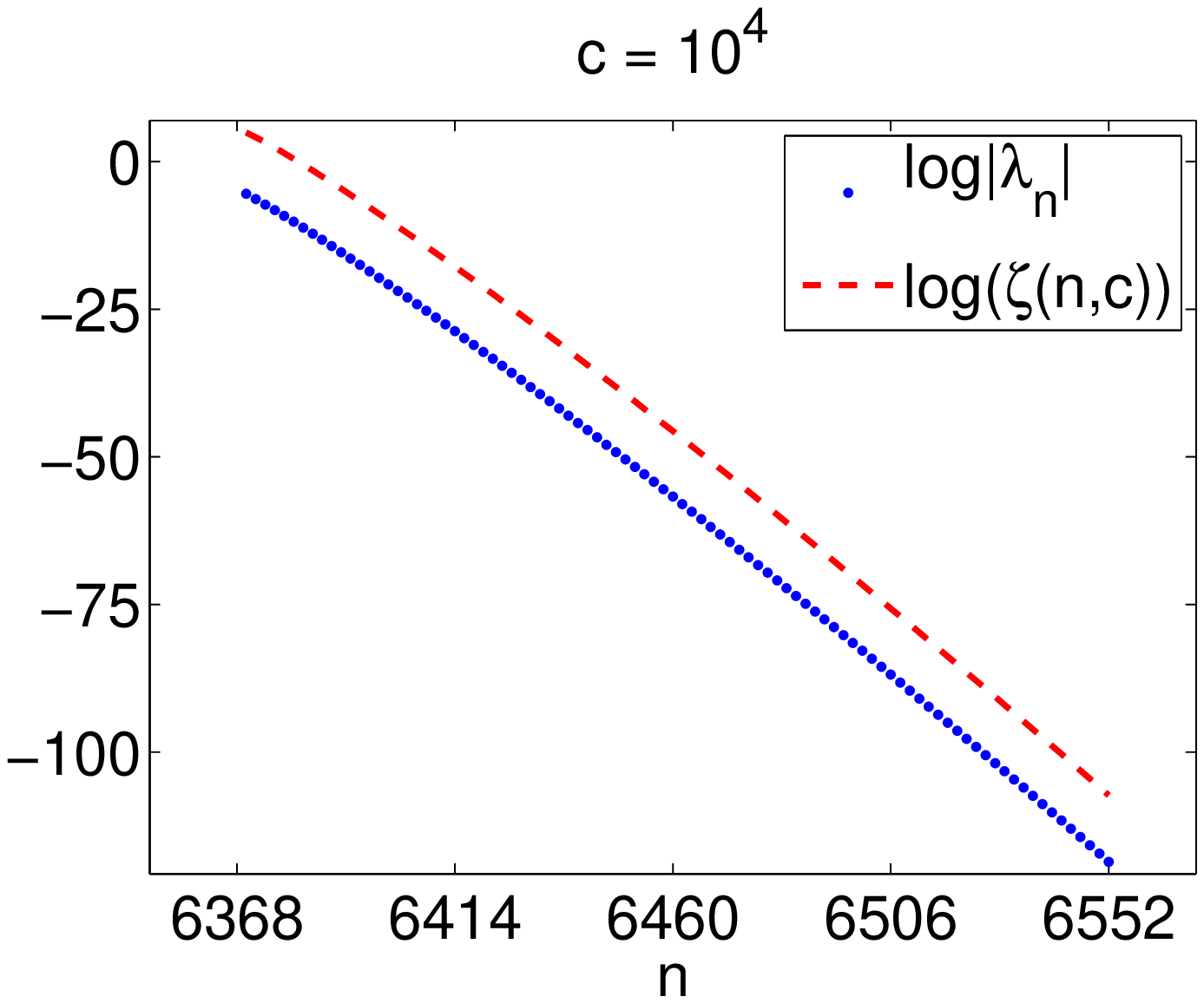}
\label{fig:test151_c10000}
}
\subfigure[$c=100,000$.]{
\includegraphics[width=11.5cm, bb=81 217 549 564, clip=true]
{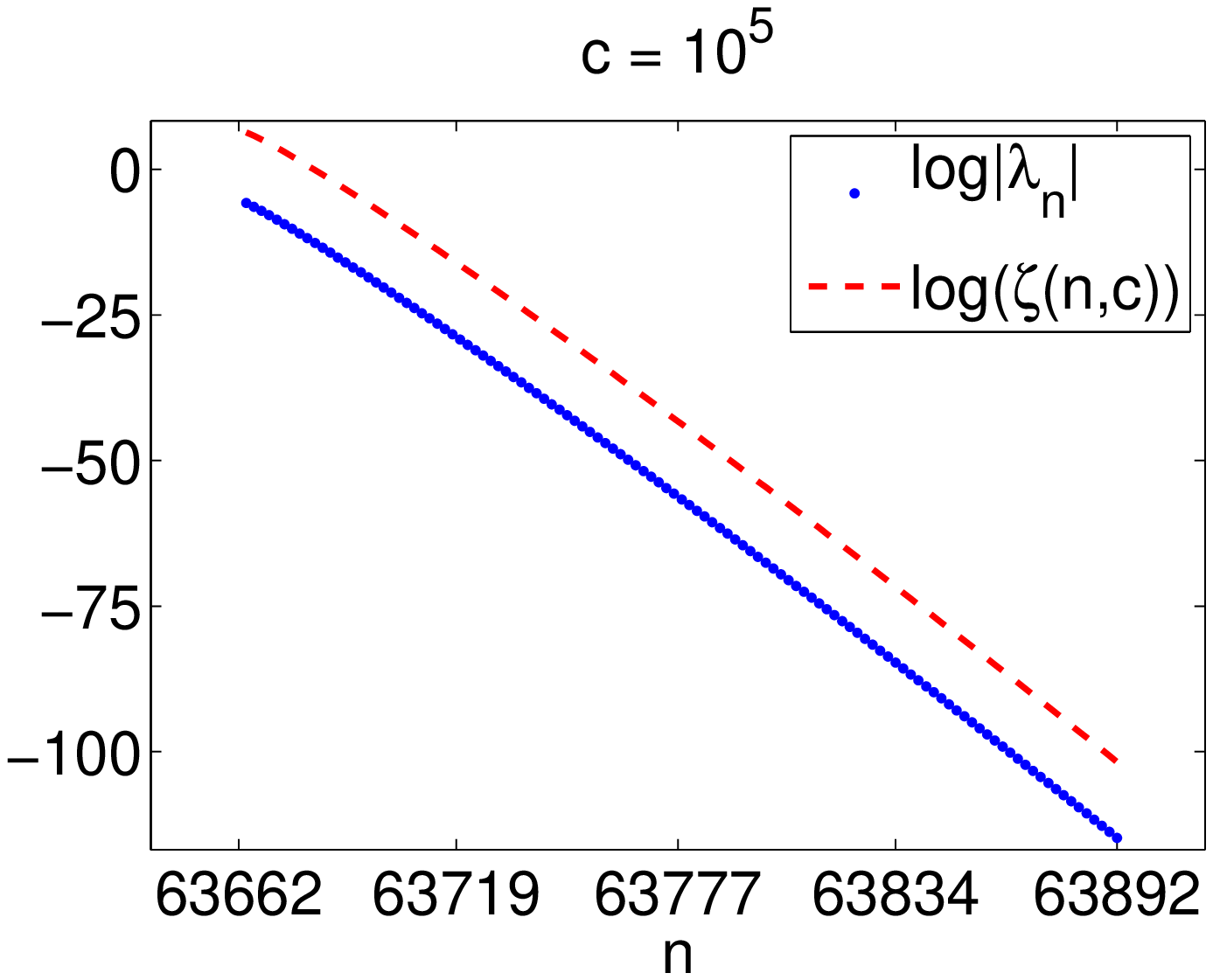}
\label{fig:test151_c100000}
}
\caption{ 
Illustration of Theorem~\ref{thm_big_inequality}. 
Corresponds to Experiment 2 in Section~\ref{sec_numerical}.
}
\label{fig:test151_2}
\end{figure}

\clearpage
\begin{figure}[ht]
\centering
\subfigure[$c=10,000$.]{
\includegraphics[width=11.5cm, bb=81 217 549 564, clip=true]
{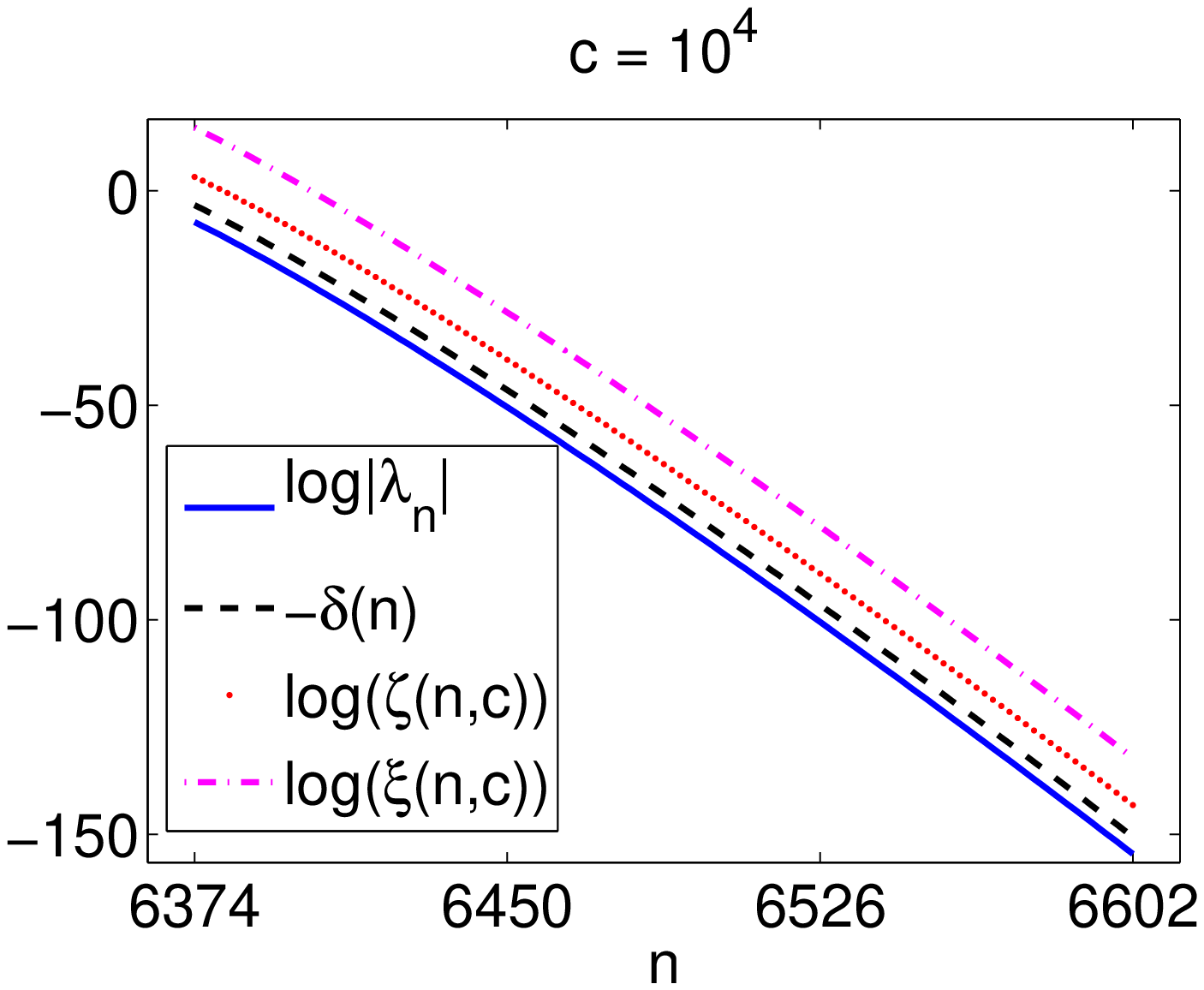}
\label{fig:test154_c10000}
}
\subfigure[$c=100,000$.]{
\includegraphics[width=11.5cm, bb=81 217 549 564, clip=true]
{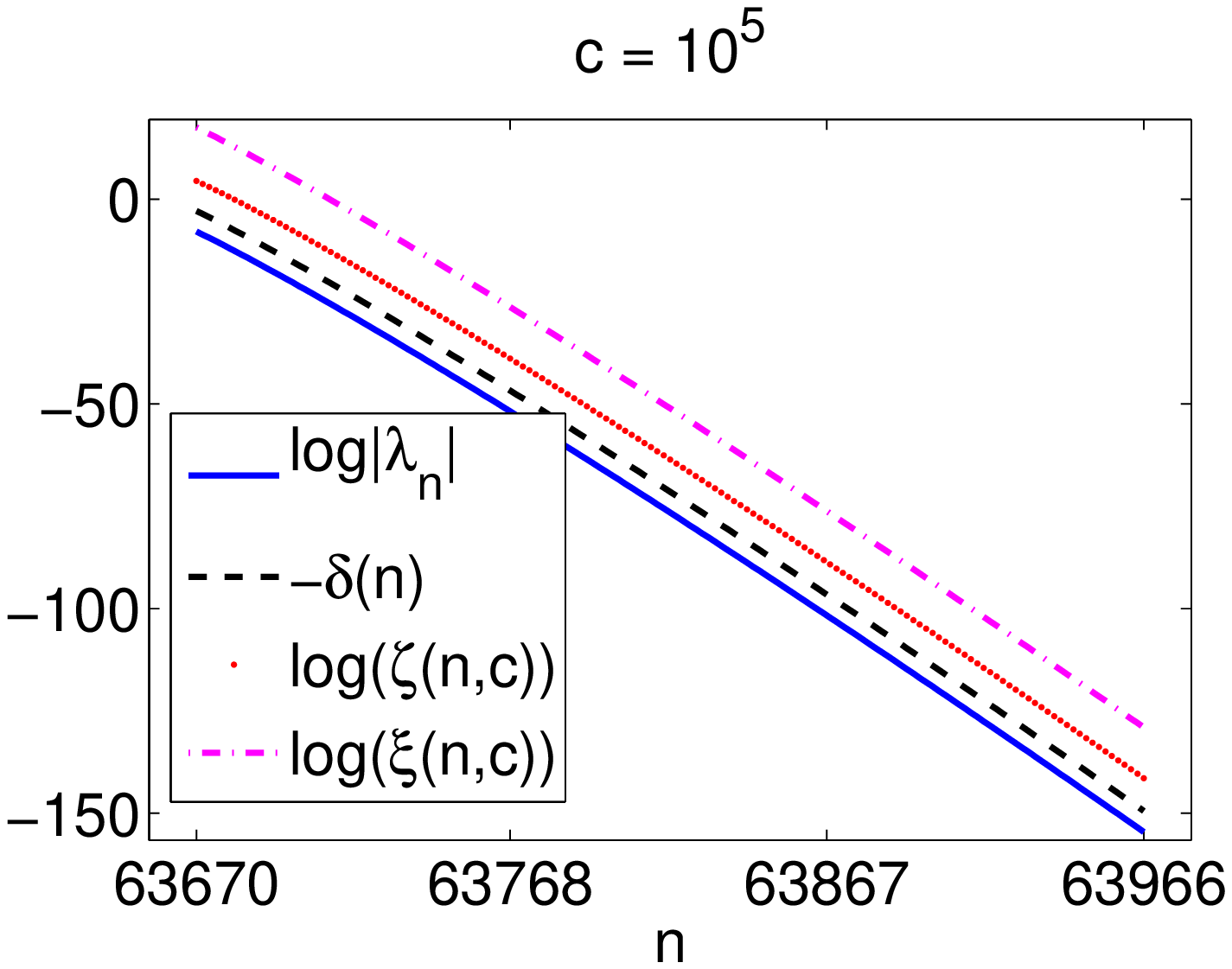}
\label{fig:test154_c100000}
}
\caption{ 
Illustration of Theorem~\ref{thm_crude_inequality}.
Corresponds to Experiment 3 in Section~\ref{sec_numerical}.
}
\label{fig:test154}
\end{figure}

\end{document}